\documentclass[a4paper,11pt]{amsart}
\usepackage{amssymb,amsmath,amsthm,color}
\usepackage{graphicx}
\usepackage{bm}
\usepackage{verbatim}
\usepackage{cite}
\usepackage{mathtools}
\usepackage{hyperref}
\numberwithin{equation}{section}

\theoremstyle{plain}
\newtheorem{thm}{Theorem}[section]

\newtheorem{prop}{Proposition}[section]
\newtheorem{lem}{Lemma}[section]

\theoremstyle{definition}
\newtheorem{rem}{Remark}[section]

\begin{document}

\title[Large-time behavior]
{Large-time behavior in an attraction--repulsion chemotaxis system}
\author[Hiroshi Wakui]{Hiroshi Wakui}
\address[Hiroshi Wakui]{Faculty of Engineering, University of Fukui, Fukui-shi, Fukui 910-8507, Japan}
\email{hwakui@u-fukui.ac.jp}
\author[Tetsuya Yamada]{Tetsuya Yamada}
\address[Tetsuya Yamada]{National Institute of Technology(KOSEN), Fukui College, Sabae, Fukui 916-8507, Japan. }
\email{yamada@fukui-nct.ac.jp}
\date{\today}

\begin{abstract}
We investigate the large-time behavior of small perturbations of stable constant steady states for an attraction-repulsion chemotaxis system in $n$-dimensional Euclidean space.
We consider integrable perturbations in one space dimension and, in higher dimensions, 
perturbations belonging to Lebesgue spaces with suitable exponents. 
We first establish decay estimates throughout the full admissible range of Lebesgue exponents and show that the nonlinear perturbation is asymptotically approximated by the corresponding linearized evolution. 
We then identify the leading term of the linearized evolution as an effective heat flow whose diffusion coefficient is explicitly determined by the parameters of the system. 
Consequently, when the integrability exponent is strictly smaller than the space dimension, 
the leading asymptotic profile is governed by the heat equation. 
At the critical endpoint, the difference between the nonlinear perturbation and the effective heat flow vanishes under the natural parabolic scaling. 
In the one-dimensional integrable case, this yields a Gaussian profile determined by the total mass of the initial perturbation.
\end{abstract}

\maketitle

%%%%%%%%%%%%%%%%%%%%%%%%%%%%%%%%%%%%%%%%%%%%%%%%%%%%%%%%%%%%%%%%%%%%
\section{Introduction} 

In this paper, 
we consider the large-time behavior of solutions to the following problem associated with an attraction-repulsion chemotaxis system:
\begin{equation}\label{P}\tag{P}
	\left\{
		\begin{aligned}
			&\partial_t u = \Delta u-\nabla\cdot(u\nabla (\beta_1B_{\lambda_1}-\beta_2 B_{\lambda_2})
			*u), 
			&&\qquad t>0, \ x \in \mathbb{R}^n, \\
			&u(0,x)=u_0(x), 
			&&\qquad x \in \mathbb{R}^n,
		\end{aligned}
	\right.
\end{equation}
where $n\ge 1$ and $\beta_j$, $\lambda_j$ $(j=1,2)$ are positive constants with $\lambda_1\ne \lambda_2$.
Here $\nabla B_{\lambda}*u$ expresses the convolution 
of $\nabla B_{\lambda}$ and $u$ with respect to $x \in \mathbb{R}^n$, that is 
\[
	(\nabla B_{\lambda}*u)(t,x)
	\coloneqq 
	\int_{\mathbb{R}^n}
		\nabla B_{\lambda}(x-y)
		u(t,y)
	\, \mathrm{d}y, 
\]
where 
\begin{align*}
	B_{\lambda}(x)
	\coloneqq \, &
	\frac{1}{(4\pi)^{\frac{n}{2}}}
	\int_0^\infty 
		\mathrm{e}^{
				-
				\lambda
				\sigma
				-
				\frac{|x|^2}{4\sigma}
			}
		\sigma^{-\frac{n}2}
	\, \mathrm{d}\sigma, 
	\qquad 
	\lambda > 0,\quad 
	x \in \mathbb{R}^n\setminus \{ 0 \}.
\end{align*} 

Problem \eqref{P} is the nonlocal formulation of the following attraction-repulsion chemotaxis system:
\begin{equation}\label{eq;at-re}\tag{ARKS}
  \left\{
  \begin{aligned}
  	\partial_t u - \Delta u \, + \, \, &\nabla \cdot \big(  u \nabla (\beta_1 \psi_1 - \beta_2 \psi_2) \big) = 0,&& \quad t>0,\quad x \in \mathbb{R}^{n},\\
 - &\Delta \psi_1 + \lambda_1 \psi_1 = u,&& \quad  t>0,\quad x \in \mathbb{R}^{n},\\
 - &\Delta \psi_2 + \lambda_2 \psi_2 = u ,&& \quad  t>0,\quad x \in \mathbb{R}^{n},\\
 &u(0,x)  = u_0(x),&& \quad x \in \mathbb{R}^{n},
 \end{aligned}
  \right.
\end{equation}
where $\beta_1, \beta_2, \lambda_1, \lambda_2 > 0$ are given constants.
Solving the second and third equations in system \eqref{eq;at-re} gives
$
	\psi_j
	=
	B_{\lambda_j}*u
	\,
	(j = 1, 2)
$, and then substituting these expressions into the first equation yields problem \eqref{P}.
In system \eqref{eq;at-re}, 
the function 
$
	u
	=
	u(t,x)
$ 
is the density of cells, 
the function 
$
	\psi_1
	=
	\psi_1(t,x)
$
represents the concentration of a chemoattractant, and 
the function 
$
	\psi_2
	=
	\psi_2(t,x)
$ 
represents the concentration of a chemorepellent.
This system \eqref{eq;at-re} 
has been studied in various frameworks
on the whole space $\mathbb{R}^n$.
For instance, the blow-up and global existence of solutions to
system \eqref{eq;at-re} have been extensively studied; see, \textit{e.g.}, \cite{Ho}, \cite{HoOg2}, 
\cite{NY2018J},
 \cite{NY2020R}, \cite{SW2015}.
See also \cite{JL2015}, \cite{NSYP1}, \cite{NSYP}, \cite{Ya1} for related results.

For every $A \in \mathbb{R}$,  
$
	( 
		u, \psi_1, \psi_2 
	) 
	= 
	( 
		A, A/\lambda_1, A/\lambda_2 
	)
$ 
is a steady state of system \eqref{eq;at-re},
and $u \equiv A$ is a steady state of problem \eqref{P}.
Let 
\begin{equation}\label{const;as}
	A_\ast
	\coloneqq 
	\left\{
	\begin{aligned}
		&\qquad \infty, 
		&&\qquad
		\dfrac{\beta_1}{\beta_2} \le 1, \quad 
		\dfrac{\beta_1}{\beta_2} \le \dfrac{\lambda_1}{\lambda_2}, \\
		&\dfrac{\lambda_1\lambda_2}{\beta_1\lambda_2-\beta_2\lambda_1}, 
		&&\qquad
		\dfrac{\beta_1}{\beta_2} > \dfrac{\lambda_1}{\lambda_2}, \quad 
		\dfrac{\beta_1}{\beta_2} \ge \left(\dfrac{\lambda_1}{\lambda_2}\right)^2, \\
		&\dfrac{\lambda_1-\lambda_2}{(\sqrt{\beta_1}-\sqrt{\beta_2})^2}, 
		&&\qquad
		\dfrac{\beta_1}{\beta_2} < \left(\dfrac{\lambda_1}{\lambda_2}\right)^2, \quad 
		\dfrac{\beta_1}{\beta_2} > 1.
	\end{aligned}
	\right.
\end{equation}
Wakui and Yamada \cite{WY} proved that every constant steady state $u \equiv A$ of problem \eqref{P} is stable in suitable Lebesgue spaces under the condition $0 < A < A_\ast$.

\begin{prop}[Theorem~2.1 in~\cite{WY}]\label{p1}
Let $p$ satisfy
\[ 
	\left\{
	\begin{aligned}
		&p =  1, &&\qquad \text{if} \quad  n=1, \\
		\frac{n}2 \, & < p \le  n, &&\qquad \text{if} \quad n \ge 2
	\end{aligned}
	\right.
\]
and let $q_0$ satisfy $n < q_0 \le 2p$.
If 
$
	0<A<A_\ast
$, 
then there exists $\varepsilon_\ast > 0$ such that 
for all $v_0 \in L^p(\mathbb{R}^n)$ with 
$
	\|
		v_0
	\|_p
	<
	\varepsilon_\ast
$, 
problem \eqref{P} 
admits a unique global solution $u$ corresponding to the initial data $u_0=A+v_0$, which satisfies the following properties:
\begin{align}
	&\label{regularity;uA} 
	u-A \in C([0,\infty);L^p(\mathbb{R}^n))\\
	&\label{decay;uA}
	\sup_{t>0}
		\|u(t)-A\|_p
	+
	\sup_{t>0}
		t^{\frac{n}2(\frac1p-\frac1{q_0})}
		\|
			u(t)
			-
			A
		\|_{q_0}
	\le 
	C_0
	\|
		u_0
		-
		A
	\|_p
\end{align}
for some $C_0 > 0$ independent of $v_0$.
\end{prop}

According to Proposition \ref{p1}, the threshold $A_*$ defined by \eqref{const;as} determines the stability range of constant steady states.
The purpose of this paper is to investigate the detailed large-time behavior of solutions to problem \eqref{P} near a stable constant steady state $u \equiv A$.
The stability result in \cite{WY} provides global solutions and decay estimates in a restricted range of Lebesgue exponents, 
but it does not identify the precise leading asymptotic profile.
The first aim of the present paper is to show that the nonlinear perturbation is asymptotic to the solution of the linearized problem.
The second aim is to identify the leading term of the linearized semigroup.
The key observation is the low-frequency expansion of the Fourier symbol of the linearized operator,
which reveals an effective heat flow.
Combining these two steps yields the asymptotic profiles stated in Theorem \ref{t3}.

\begin{thm}\label{t1}
Let $p$, $A$ and $v_0$ satisfy the assumptions of Proposition {\rm \ref{p1}}, and 
let $u$ be the corresponding solution of problem \eqref{P} with the initial data $u_0 \coloneqq A+v_0$, 
satisfying \eqref{regularity;uA} and \eqref{decay;uA}. 
Then, for every $p \le q \le \infty$, there exist $C_1, C_2 > 0$ independent of $v_0$ such that 
\begin{equation}\label{decay;ua}
	\sup_{t>0}t^{\frac{n}2(\frac1p-\frac1q)}
	\|
		u(t)-A
	\|_q
	\le 
	C_1
	\|
		u_0-A
	\|_p
	,
\end{equation}
\begin{equation}\label{asy;ua}
	\sup_{t>0}
	t^{\frac{n}2(\frac1p-\frac1q)+\frac{n}{2}(\frac1p-\frac1n)}
		\|
			u(t)-A-\mathcal{U}(t)
		\|_q
	\le 
	C_2
	\|
		u_0-A
	\|_p
	.
\end{equation}
Here $\mathcal{U}(t,x)$ denotes the solution of the following linearized problem on $(0,\infty) \times \mathbb{R}^n$: 
\begin{equation}\label{eq;calU}
	\left\{
	\begin{aligned}
	&\partial_t\mathcal{U}=\Delta \mathcal{U}-A\Delta(\beta_1B_{\lambda_1}-\beta_2B_{\lambda_2})*\mathcal{U}, 
	&&\qquad t>0, \ x\in \mathbb{R}^n,\\
	&\mathcal{U}(0,x)
	=
	u_0(x)-A, 
	&&\qquad x \in \mathbb{R}^n, 
	\end{aligned}
	\right.
\end{equation} 
where the symbol $\ast$ stands for the convolution  with respect to $x\in \mathbb{R}^n$.
\end{thm}

\begin{rem}
It follows from \cite[Chapter I\hspace{-.2em}I\hspace{-.2em}I, Theorem 2.10]{EN2000} that 
the closure of the operator 
$
	\mathcal{L}_A
	\coloneqq 
	\Delta
	-
	A
	\Delta
	(
		\beta_1
		B_{\lambda_1}
		-
		\beta_2
		B_{\lambda_2}
	)*
$ 
generates an analytic semigroup on 
$L^q(\mathbb{R}^n)$ for every $1 \le q < \infty$. 
Moreover, this analytic semigroup admits the following Fourier multiplier representation: 
\begin{equation}\label{eq;semila}
	\mathrm{e}^{t\mathcal{L}_A}
	f
	\coloneqq 
	\mathcal{F}^{-1}
	\big[ 
		\mathrm{e}^{th_A}\mathcal{F}[f] 
	\big]
	=
	(2\pi)^{-\frac{n}2}
	\mathcal{F}^{-1}
	[ 
		\mathrm{e}^{th_A}
	]
	*
	f
	, \quad t>0, 
\end{equation}
where $\mathcal{F}^{-1}$ denotes the inverse Fourier transform and 
$h_A(\xi)$ is the function defined by 
\begin{equation}
\label{fn;ha}
	h_A(\xi)
	\coloneqq 
	-|\xi|^2
	+
	\beta_1A
	\frac{|\xi|^2}{\lambda_1+|\xi|^2}
	-
	\beta_2A
	\frac{|\xi|^2}{\lambda_2+|\xi|^2}. 
\end{equation}
Consequently, the solution $\mathcal{U}(t,x)$ of \eqref{eq;calU} is given by
\[
	\mathcal{U}(t,x)
	=
	\mathrm{e}^{t\mathcal{L}_A}(u_0-A)(x).
\]
For $f \in L^\infty(\mathbb{R}^n)$, $\mathrm{e}^{t\mathcal{L}_A}f$ denotes the corresponding convolution operator defined by its $L^1$-kernel.
\end{rem}

Theorem \ref{t1} shows that $u(t)-A$ behaves like $\mathcal{U}(t)$ introduced  by \eqref{eq;calU} in the case $n/2<p<n$ as $t \to \infty$ since 
\[
	\sup_{t>0}
	t^{\frac{n}2(\frac1p-\frac1q)}
	\|
		\mathcal{U}(t)
	\|_q
	\le 
	C
	\|
		u_0-A
	\|_p
\] 
by \eqref{est;lrlp} below.
The next theorem identifies the leading term of the linearized semigroup $\mathrm{e}^{t\mathcal{L}_A}$.

\begin{thm}\label{t2}
Let $1 \le p \le q \le \infty$ and $0<A<A_\ast$. 
Then there exists $C = C(n, \beta_1, \beta_2, \lambda_1, \lambda_2, A, p, q)>0$ such that 
\begin{align}
	\label{asy;la}
	&
	\|
		\mathrm{e}^{t\mathcal{L}_A}f
		-
		\mathrm{e}^{c_{\ast} t\Delta}f
	\|_q
	\le 
	C
	t^{-\frac{n}2(\frac1p-\frac1q)}
	(1+t)^{-1}
	\|
		f
	\|_p 
\end{align}
holds for all $t>0$ and $f\in L^p(\mathbb{R}^n)$,
where
\begin{equation}\label{const;ca}
	c_{\ast}
	\coloneqq 
	1
	-
	\frac{\beta_1\lambda_2-\beta_2\lambda_1}{\lambda_1\lambda_2}A>0
\end{equation}
and  $\mathrm{e}^{c_{\ast} t\Delta}$ is the heat semigroup.
\end{thm}

Combining Theorems \ref{t1} and \ref{t2} yields the following refinement of \eqref{asy;ua}.
More precisely, Theorem \ref{t1} extends the decay estimates to the full range of Lebesgue exponents, 
whereas Theorem \ref{t2} shows that the difference between the linearized evolution and the effective heat flow has an additional decay factor of order $(1+t)^{-1}$.

\begin{thm}\label{t3}
Let  $p$, $A$ and $v_0$ satisfy the assumptions of Proposition \ref{p1}.
If 
the solution $u$ to problem \eqref{P} with the initial data $u_0\coloneqq A+v_0$ satisfies \eqref{regularity;uA} and \eqref{decay;uA}, 
then the following two assertions hold:   
\begin{enumerate}
\item If $n\ge 2$ and $n/2<p<n$, then
\begin{equation}\label{asy;ua2}
\begin{split}
	t^{\frac{n}2(\frac1p-\frac1q)}
	\|
		u(t)-A
		-
		\mathrm{e}^{c_{\ast} t\Delta}(u_0-A)
	\|_q 
	\le 
	C
	\big(
		t^{-1}
		+
		t^{-\frac{n}2(\frac1p-\frac1n)}
	\big)
	\|
		u_0-A
	\|_p
\end{split}
\end{equation}
for all $t>0$ and $p \le q \le \infty$,
where $C>0$ depends only on $n,\beta_1, \beta_2, \lambda_1, \lambda_2, p, q$ and $A$.
\item 
Let $\varepsilon_* > 0$ be the same constant as in Proposition \ref{p1}.
For every $n \le q \le \infty$,
there exists 
$
	\varepsilon_0 
	=
	\varepsilon_0(q)
	\in 
	( 
		0, \min \{ 1, \varepsilon_* \}
	]
$ 
such that,
if 
$n \ge 1$, 
$p=n$
and 
$
	\|
		u_0-A
	\|_n
	<
	\varepsilon_0
$, then 
\begin{equation}\label{asy;ua1}
	\lim_{t\to\infty}
	t^{\frac{n}2(\frac1n-\frac1q)}
	\|
		u(t)-A
		-
		\mathrm{e}^{c_{\ast} t\Delta}(u_0-A)
	\|_q
	=
	0,
\end{equation}
where the constant $c_{\ast}>0$ is defined by \eqref{const;ca}, and $\mathrm{e}^{c_\ast t\Delta}$ is the heat semigroup.
\end{enumerate}

\end{thm}

\begin{rem}
It is well known that if $f\in L^1(\mathbb{R}^n)$,
then, for all $1 \le q \le \infty$, 
\[
	\lim_{t \to \infty}
	t^{\frac{n}2(1-\frac1q)}
	\left\|
		\mathrm{e}^{c_\ast t\Delta}f
		-
		G(c_{\ast} t)
		\int_{\mathbb{R}^n}
			f
		\, \mathrm{d}x
	\right\|_q
	=
	0,
\]
where $G(t,x) \coloneqq (4\pi t)^{-n/2}\mathrm{e}^{-|x|^2/(4t)}$;
see, \textit{e.g.}, \cite{G2010}. 
Hence, if $p = n = 1$,
then, for every fixed $1 \le q \le \infty$ and every perturbation satisfying
$
	\|
		u_0
		-
		A
	\|_1
	<
	\varepsilon_0
$, 
where $\varepsilon_0 = \varepsilon_0(q) > 0$ is the threshold in Theorem \ref{t3} (\textrm{ii}),
the decay estimate in \eqref{asy;ua1} can be modified as follows:
\begin{equation}\label{asy;mua1}
	\lim_{t\to\infty}
	t^{\frac12(1-\frac1q)}
	\left\|
		u(t)-A
		-
		G(c_{\ast} t)
		\int_{\mathbb{R}}
			(u_0-A)
		\, \mathrm{d}x
	\right\|_q
	=
	0.
\end{equation}
As a result, we observe from \eqref{decay;ua} and \eqref{asy;mua1} that 
there are constants $c$, $C>0$ such that for all sufficiently large $t$ and $1\le q\le \infty$, 
\[
	c
	t^{-\frac{1}{2}(1-\frac{1}{q})}
	\le 
	\|
		u(t)
		-
		A
	\|_q
	\le 
	C
	t^{-\frac{1}{2}(1-\frac{1}{q})}
\] 
when $n=p=1$ and $\int_{{\mathbb R}^n}(u_0-A)\,{\rm d}x \ne 0$. 
This means that the time decay rate in \eqref{decay;ua} is optimal in this case.
\end{rem}

\begin{rem}
	The conclusions of Theorems~\ref{t1}--\ref{t3} also hold, under the corresponding assumptions, for the following chemotaxis problem with a single chemoattractant:
	\begin{equation}\label{sat}\tag{AT}
	\left\{
		\begin{aligned}
			&\partial_t u = \Delta u - \beta \nabla \cdot ( u \nabla B_{\lambda}*u), 
			&&\qquad t>0, \ x \in \mathbb{R}^n, \\
			&u(0,x)=u_0(x), 
			&&\qquad x \in \mathbb{R}^n,
		\end{aligned}
	\right.
\end{equation}
where $\beta, \lambda > 0$.
In this case, $A_*$ and $c_*$ are to be replaced by $\lambda/\beta$ and $1 - \beta A /\lambda$, respectively.
\end{rem}

We now briefly explain the strategy to prove our theorems.
Fix a constant steady state $u \equiv A$ of problem \eqref{P}.
Setting 
$
	u
	\coloneqq 
	A
	+
	v
$ 
transforms problem \eqref{P} into the following perturbation problem:
\begin{equation}\label{eq;Q}\tag{Q}
	\left\{
	\begin{aligned}
		&\partial_tv-\mathcal{L}_Av=-\nabla \cdot (v\nabla K*v), 
		&& \qquad t>0, \ x\in \mathbb{R}^n, \\
		&v(0,x) = u_0(x)-A \eqqcolon v_0(x), 
		&&\qquad x\in \mathbb{R}^n, 
	\end{aligned}
	\right.
\end{equation}
where $K(x)=\beta_1B_{\lambda_1}(x)-\beta_2B_{\lambda_2}(x)$ and
$
\mathcal{L}_A\coloneqq \Delta - A\Delta K*
$.
The following proposition gives a unique global solution to problem \eqref{eq;Q}: 
\begin{prop}[Theorem~8.3 in~\cite{WY}]\label{prop;qgs}
Suppose that the system parameters, $p$, and $A$ satisfy the assumptions of Proposition \ref{p1}. 
Then there exists $\varepsilon_\ast > 0$ such that 
for all $v_0 \in L^p(\mathbb{R}^n)$ with $\|v_0\|_p<\varepsilon_\ast$, problem \eqref{eq;Q}  
admits a unique global solution $v$ in the sense that 
\begin{align}
	&\label{R;solv} 
	v
	\in 
	C([0,\infty);L^p(\mathbb{R}^n)) \cap C((0,\infty); L^{2p}(\mathbb{R}^n));\\
	&\label{eq;ieqv} 
	v(t)
	=
	\mathrm{e}^{t\mathcal{L}_A}v_0
	-
	\int_0^t
		\nabla \cdot \mathrm{e}^{(t-s)\mathcal{L}_A}
		(
			v
			\nabla 
			K
			*v
		)(s)
	\, \mathrm{d}s, \quad t>0.
\end{align}
Moreover, the following decay properties hold: for all $n<q\le 2p$, 
\begin{equation}\label{dEcay;v0}
	\sup_{t>0}
	\|
		v(t)
	\|_p
	+
	\sup_{t>0}
	t^{\frac{n}2(\frac1p-\frac1q)}
	\|
		v(t)
	\|_q
	\le 
	C
	\|
		v_0
	\|_p, 
\end{equation}
where $C>0$ is independent of $v_0$ and $t$, 
and $\mathrm{e}^{t\mathcal{L}_A}$ is the analytic semigroup 
on $L^q(\mathbb{R}^n)$ $(1\le q<\infty)$ denoted by \eqref{eq;semila}. 
\end{prop}

By replacing the smallness constants in Propositions \ref{p1} and \ref{prop;qgs} with their minimum,
we use the same notation $\varepsilon_*$ for both constants.
Decreasing $\varepsilon_*$ further if necessary, we 
may assume that $0 < \varepsilon_* \le 1$.
Since Proposition \ref{prop;qgs} yields a unique global solution $v$ to problem \eqref{eq;Q} satisfying \eqref{R;solv}--\eqref{dEcay;v0} under these assumptions, by setting 
$
	u
	=
	v
	+
	A
$,  
we obtain a unique global solution $u$ to problem \eqref{P} satisfying \eqref{regularity;uA} and \eqref{decay;uA}.
Hence we only need to investigate the detailed large-time behavior of the global solution $v$ to problem \eqref{eq;Q} obtained in Proposition \ref{prop;qgs}.

Section 2 presents preliminary estimates used throughout the paper.
Section 3 is devoted to the proof of Theorem \ref{t1}.
Section 4 establishes the approximation of the linearized semigroup by the effective heat semigroup and proves Theorem \ref{t2}.
Section 5 combines these results to prove Theorem \ref{t3}.

Before closing this section, we 
introduce notation used throughout this paper. 

\subsection{Notation} 
Let 
$
	\mathbb{Z}_+
$ 
denote
the set of nonnegative integers. 
For 
$
	\alpha
	=
	(
		\alpha_1, 
		\alpha_2,
		\ldots,
		\alpha_n
	)
	\in 
	\mathbb{Z}_+^n
$, 
set 
$
	|\alpha|
	\coloneqq
	\alpha_1
	+
	\alpha_2
	+
	\cdots
	+
	\alpha_n
$. 
We write  
$
	\partial_t
	=
	{\partial}/{\partial t}
$,
$
	\partial_j
	=
	{\partial}/{\partial x_j}
$ 
and
$
	\nabla
	=
	{}^t
	(
		\partial_1, 
		\partial_2,
		\ldots, 
		\partial_n
	)
$.
Moreover,
$
 \partial_x^\alpha
 =
 \partial_1^{\alpha_1}
 \partial_2^{\alpha_2}
 \cdots
 \partial_n^{\alpha_n}
$ 
for 
$
 \alpha
 =
 (
  \alpha_1,
  \alpha_2,
  \ldots, 
  \alpha_n
 )
 \in 
 \mathbb{Z}_{+}^n
$. 
For 
$
 1
 \le 
 p
 \le 
 \infty
$, 
the norm of the usual Lebesgue space $L^p(\mathbb{R}^n)$ is denoted by $\|\cdot\|_p$. 
Let 
$
	\mathcal{S}
	=
	\mathcal{S}(\mathbb{R}^n)
$ 
be the set of rapidly decreasing functions on $\mathbb{R}^n$. 
For $f \in \mathcal{S}$, the symbol $\mathcal{F}[f]$ stands for the Fourier transform of $f$, that is, 
$
 \mathcal{F}[f](\xi)
 \coloneqq 
 (2\pi)^{-n/2}
 \int_{\mathbb{R}^n} \mathrm{e}^{-ix \cdot \xi}
  f(x)
 \, \mathrm{d}x, 
 \xi
 \in \mathbb{R}^n
$, 
and the inverse Fourier transform of $f$ is defined by 
$
 \mathcal{F}^{-1}[f](x)
 \coloneqq 
 (2\pi)^{-n/2}
 \int_{\mathbb{R}^n} 
  \mathrm{e}^{ix \cdot \xi}
   f(\xi)
  \, \mathrm{d}\xi
 , x \in \mathbb{R}^n
$. 
For $x \in \mathbb{R}$,
let
$
	\lfloor 
		x
	\rfloor
$
denote the greatest integer less than or equal to $x$.
Let $C$ be a positive constant which may change from line to line. In particular, we write $C(\ast,\ldots,\ast)$ for a positive constant depending on the quantities appearing in parentheses.

\section{Preliminaries}

In this section, we collect several preliminary estimates used in Theorems \ref{t1}, \ref{t2} and \ref{t3}.
We begin with the following lemma:

\begin{lem}\label{lem;ede}
Let $n \ge 1$, $s > -n$ and 
$
	k > 0
$. 
Then there exists $C = C(n,s,k) > 0$ such that 
\begin{equation}\label{es;ede}
	\int_{\mathbb{R}^n} 
		|\xi|^s\mathrm{e}^{-kt |\xi|^2}
	\,
	\mathrm{d}\xi
	=
	Ct^{-\frac{n}2-\frac{s}2}, \quad t>0.
\end{equation}
\end{lem}

\begin{proof}
By the change of variables $\eta = (kt)^{1/2} \xi$, we obtain 
\begin{align*}
	\int_{\mathbb{R}^n} 
		|\xi|^s
		\mathrm{e}^{-kt|\xi|^2}
	\, \mathrm{d}\xi
	= \, &
	\int_{\mathbb{R}^n} 
		|(kt)^{-\frac12}\eta|^s
		\mathrm{e}^{-|\eta|^2}(kt)^{-\frac{n}2}
	\, \mathrm{d}\eta \\
	= \, &
	(kt)^{-\frac{n}2-\frac{s}2}
	\int_{\mathbb{R}^n} 
		|\eta|^s
		\mathrm{e}^{-|\eta|^2}
	\, \mathrm{d}\eta \\
	= \, &
	C
	t^{-\frac{n}2-\frac{s}2}, 
	\qquad t>0.
\end{align*}
Therefore we obtain the desired estimate \eqref{es;ede}. 
\end{proof}
By \cite[Lemma 3.6]{CKKW}, we get properties of the function $K(x)=\beta_1B_{\lambda_1}(x)-\beta_2B_{\lambda_2}(x)$. 
\begin{lem}\label{lem;bessel}
Let $\alpha \in \mathbb{Z}_{+}^n$ with $|\alpha| \le 1$ 
and assume that 
\[
	1
	\le 
	q
	<
	\frac{n}{n-2+|\alpha|} 
	\quad 
	\text{if} \quad n \ge 2 \quad 
	\text{or} \quad 1 \le q \le \infty \quad \text{if} \quad n=1.
\]
Then $\partial_x^\alpha K \in L^q(\mathbb{R}^n)$.
\end{lem}

\begin{rem}
The condition on $q$ is understood as $1\le q<\infty$ when $n=2$ and $|\alpha|=0$ in Lemma \ref{lem;bessel}.
\end{rem}

The following proposition gives the $L^p$-$L^q$ estimates for the analytic semigroup 
$
	\mathrm{e}^{t\mathcal{L}_A}
$ 
defined by \eqref{eq;semila}.

\begin{prop}[Proposition~6.6 in~\cite{WY}]\label{prop;lrlpest}
Let $1 \le p \le q \le \infty$, and suppose that $0 < A < A_*$. 
Then, for all $t>0$ and $f\in L^p(\mathbb{R}^n)$, 
\begin{align}
	& \,
	\label{est;lrlp}
	\|
		\mathrm{e}^{t\mathcal{L}_A}
		f
	\|_q
	\le 
	C
	t^{-\frac{n}2(\frac1p-\frac1q)}
	\|
		f
	\|_p, \\
	&\label{est;nlrlp}
	\|
		\nabla \mathrm{e}^{t\mathcal{L}_A}
		f
	\|_q
	\le 
	C'
	t^{-\frac{n}2(\frac1p-\frac1q)-\frac12}
	\|
		f
	\|_p
\end{align}
hold,
where 
$
	C 
	= 
	C(n,\beta_1,\beta_2,\lambda_1,\lambda_2,A,p,q) 
	> 
	0
$ 
and 
$
	C'
	=
	C'(n,\beta_1,\beta_2,\lambda_1,\lambda_2,A,p,q) 
	>
	0
$
are independent of $f$ and $t$.
\end{prop}

\section{Proof of Theorem \ref{t1}}
Let $p$, $A$ and $v_0$ satisfy
the same conditions as in Proposition \ref{p1}, 
and let $v$ be the global solution to problem \eqref{eq;Q} constructed in Proposition \ref{prop;qgs}. 
As mentioned in the introduction, we assume that the constant $\varepsilon_\ast$ in Proposition \ref{prop;qgs} satisfies 
$0<\varepsilon_\ast\le 1$.
We first improve the decay estimate in \eqref{dEcay;v0} for $n=1$.

\begin{lem}\label{lem;idecayv1}
Let $n=1$ and $1\le q\le \infty$. Then we have
\begin{equation}\label{decay;v1}
	\|
		v(t)
	\|_q
	\le 
	C
	t^{-\frac12(1-\frac1q)}
	\|v_0\|_1
	, \qquad t>0,
\end{equation}
where $C > 0$ depends only on $\beta_1, \beta_2, \lambda_1, \lambda_2, A$ and $q$.
\end{lem}

\begin{proof}
Taking the $L^q$-norm in \eqref{eq;ieqv} 
and applying \eqref{est;lrlp}, we obtain
\begin{equation}\label{no;v0}
	\begin{aligned}
		\|
			v(t)
		\|_q
		\le \, &
		\|
			\mathrm{e}^{t\mathcal{L}_A}v_0
		\|_q
		+
		\|
			I(t)
		\|_q\\
		\le \, &
		C
		t^{-\frac12(1-\frac1q)}
		\|
			v_0
		\|_1
		+
		\|
			I(t)
		\|_q, 
		\qquad t>0, 
	\end{aligned}
\end{equation}  
where 
\begin{equation}\label{eq;II}
	I(t) 
	\coloneqq 
	-\int_0^t
		\nabla \cdot \mathrm{e}^{(t-s)\mathcal{L}_A}(v\nabla K*v)(s)
	\, \mathrm{d}s.
\end{equation}

We first estimate the $L^4$-norm for $v(t)$.
Using \eqref{dEcay;v0} with $q=2$, \eqref{est;nlrlp}, Lemma \ref{lem;bessel}, 
Young's 
inequality and the Cauchy-Schwarz inequality, we obtain  
\begin{equation}\label{decay;I4}
	\begin{aligned}
		\|
			I(t)
		\|_4
		\le \, &
		C
		\int_0^t
			(t-s)^{-\frac78}
			\|
				v(s)
				\nabla K*v(s)
			\|_1
		\, \mathrm{d}s \\
		\le \, & 
		C
		\int_0^t
			(t-s)^{-\frac78}
			\|
				v(s)
			\|_2
			\|
				\nabla K*v(s)
			\|_2
		\, \mathrm{d}s \\
		\le \, & 
		C
		\int_0^t
			(t-s)^{-\frac78}
			\|
				\nabla 
                K
			\|_1
			\|
				v(s)
			\|_2^2
		\, \mathrm{d}s \\
		\le \, & 
		C
		\|
			v_0
		\|_1^2
		\int_0^t
			(t-s)^{-\frac78}
			s^{-\frac12}
		\, \mathrm{d}s \\
		\le \, & 
		C
		t^{-\frac38}
		\|
			v_0
		\|_1^2
		, \qquad t>0.
\end{aligned}
\end{equation}
Hence, plugging \eqref{decay;I4} into \eqref{no;v0} with $q=4$ implies 
\begin{equation}\label{decay;v4}
	\begin{aligned}
	\|
		v(t)
	\|_4
	\le \, & 
	C
	t^{-\frac{3}{8}}
	(
		\|
			v_0
		\|_1
		+
		\|
			v_0
		\|_1^2
	)\\
	\le \, & 
	C
	t^{-\frac{3}{8}}
		\|
			v_0
		\|_1
	, 
	\qquad t>0.
	\end{aligned}
\end{equation}

We next 
prove the 
endpoint 
estimate 
\eqref{decay;v1} 
for $q=\infty$. 
By \eqref{est;nlrlp}, \eqref{decay;v4}, Lemma \ref{lem;bessel}, 
Young's 
inequality and the Cauchy-Schwarz inequality, we get 
\begin{equation}\label{decay;Ii}
	\begin{aligned}
		\|
			I(t)
		\|_\infty
		\le \, & 
		C
		\int_0^t
			(t-s)^{-\frac34}
			\|
				v(s)
                \nabla K*v(s)
			\|_2
		\, \mathrm{d}s \\
		\le \, & 
		C
		\int_0^t
			(t-s)^{-\frac34}
			\|
				v(s)
			\|_{4}
			\|
                  \nabla K*v(s)
			\|_4
		\, \mathrm{d}s \\
		\le \, & 
		C
		\int_0^t
			(t-s)^{-\frac34}
			\|
				\nabla K
			\|_1
			\|
				v(s)
			\|_{4}^2
		\, \mathrm{d}s \\
		\le \, & 
		C
		(
			\|
				v_0
			\|_1
			+
			\|
				v_0
			\|_1^2
		)^2
		\int_0^t
			(t-s)^{-\frac34}
			s^{-\frac34}
		\, \mathrm{d}s \\
		\le \, & 
		C
		t^{-\frac12}
		(
			\|
				v_0
			\|_1
			+
			\|
				v_0
			\|_1^2
		)^2\\
		\le \, & 
		C
		t^{-\frac12}
			\|
				v_0
			\|_1
		, \qquad t>0.
\end{aligned}
\end{equation}
Therefore, decay estimate \eqref{decay;v1} with $q=\infty$ 
follows from \eqref{no;v0} with $q=\infty$ and \eqref{decay;Ii}.

As a consequence, 
we can obtain the desired estimate \eqref{decay;v1} for all $1\le q\le \infty$ 
because of 
the $L^1$-bound in \eqref{dEcay;v0}, 
\eqref{decay;v1} with $q=\infty$, 
and 
$
	\|
		f
	\|_q
	\le 
	\|
		f
	\|_1^{1/q}
	\|
		f
	\|_\infty^{1-1/q}
$
for
$1<q<\infty$.
\end{proof}

We next extend the decay estimate in \eqref{dEcay;v0} to the range $p \le q \le \infty$ for $n\ge 2$.

\begin{lem}\label{lem;idecayv2}
Let $n \ge 2$, $n/2 < p \le n$, and $p \le q \le \infty$. Then we have
\begin{equation}\label{decay;v2}
	\|
		v(t)
	\|_q
	\le 
	C
	t^{-\frac{n}2(\frac1p-\frac1q)}
	\|
		v_0
	\|_p
	, \qquad t>0,
\end{equation}
where $C>0$ depends only on $n, \beta_1, \beta_2, \lambda_1, \lambda_2, A, p$ and $q$.
\end{lem}

\begin{proof}
Let $n\ge 2$, $n/2<p\le n$ and $p\le q\le \infty$. 
Then, taking the $L^q$-norm in \eqref{eq;ieqv} and using 
\eqref{est;lrlp}, 
we have 
\begin{equation}\label{no;vr}
\begin{aligned}
	\|
		v(t)
	\|_q
	\le \, &
	\|
		\mathrm{e}^{t\mathcal{L}_A}v_0
	\|_q
	+
	\|
		I(t)
	\|_q\\
	\le \, &
	C
	t^{-\frac{n}2(\frac1p-\frac1q)}
	\|
		v_0
	\|_p
	+
	\|
		I(t)
	\|_q, \qquad t>0,
\end{aligned}
\end{equation}  
where the function $I(t)$ is the one defined by \eqref{eq;II}.  

We first prove that 
 \begin{equation}\label{decay;mur}
    \|
		v(t)
	\|_q
	\le 
	C
	t^{-\frac{n}2(\frac1p-\frac1q)}
	\big(
		\| v_0 \|_p
		+
		\| v_0 \|_p^2
	\big)
\end{equation}
for all $t>0$ and $n<q<\infty$. 
It follows from \eqref{est;nlrlp} and H\"{o}lder's inequality
that 
\begin{equation}\label{decay;Ir}
\begin{aligned}
	\|
		I(t)
	\|_q 
	\le \, &
	\int_0^t
		\left\|
			\nabla 
			\cdot 
			\mathrm{e}^{(t-s)\mathcal{L}_A}(v \nabla K*v)(s)
         \right\|_q
	\, \mathrm{d}s \\
	\le \, &
	C
	\int_0^t
		(t-s)^{-\frac{n}2(\frac1n-\frac1q)-\frac12}
		\|
			v(s)
            \nabla K*v(s)
		\|_n
	\, \mathrm{d}s \\
	\le \, &
	C
	\int_0^t
		(t-s)^{-1+\frac{n}{2q}}
		\|
			v(s)
		\|_{2p}
		\|
			\nabla K*v(s)
		\|_{\frac{2np}{2p-n}}
	\, \mathrm{d}s 
\end{aligned}
\end{equation}
for all $t>0$. Here, take $q_1$ as 
\begin{equation}\label{exp;r1}
	q_1
	=
	\frac{np}{(n+1)p-n}
	\in 
	\bigg[
		1,\frac{n}{n-1}
	\bigg), 
\end{equation}
which satisfies 
$
	(2p-n)/(2np)
	=
	1/q_1
	+
	1/(2p)
	-
	1
$. 
Using Young's inequality and Lemma \ref{lem;bessel}, 
we obtain  
\begin{equation}\label{est;nt0}
		\|
              \nabla K*v(s)
		\|_{\frac{2np}{2p-n}}
		\le 
		\|
            \nabla K
		\|_{q_1}
		\|
			v(s)
		\|_{2p}\\
		\le 
		C
		\|
			v(s)
		\|_{2p}, \qquad s>0.
\end{equation}
Hence,  
we substitute \eqref{est;nt0} into \eqref{decay;Ir} and apply \eqref{dEcay;v0} with $q=2p$ to obtain 
\begin{align*}
	\|
		I(t)
	\|_q 
	\le \, &
	C
	\int_0^t
		(t-s)^{-1+\frac{n}{2q}}
		\|
			v(s)
		\|_{2p}^2
	\, \mathrm{d}s \\
	\le \, &
	C
	\|
		v_0
	\|_p^2
	\int_0^t
		(t-s)^{-1+\frac{n}{2q}}
		s^{-\frac{n}{2p}}
	\, \mathrm{d}s\\
	\le \, &
	C
	t^{-\frac{n}2(\frac1p-\frac1q)}
	\|
		v_0
	\|_p^2
	, \qquad t>0.
\end{align*}
This together with \eqref{no;vr} gives decay estimate \eqref{decay;mur}.

We next prove decay estimate \eqref{decay;v2} with $q = \infty$ in the case  
$n/2<p<n$. Noting that 
$
	n
	<
	np/(n-p)
	<
	\infty
$ 
if $n/2<p<n$ and then using \eqref{est;nlrlp}, \eqref{decay;mur}, Lemma \ref{lem;bessel}, 
H\"{o}lder's inequality and Young's inequality, 
we get 
\begin{equation}\label{decay;nIip}
	\begin{aligned}
		\|
			I(t)
		\|_\infty
		\le \, & 
		C
		\int_0^t
			(t-s)^{-\frac{n}{4p}-\frac12}
			\|
				v(s)
                \nabla K*v(s)
			\|_{2p}
		\, \mathrm{d}s \\
		\le \, & 
		C
		\int_0^{t}
			(t-s)^{-\frac{n}{4p}-\frac12}
			\|
				v(s)
			\|_{2p}
			\|
            	\nabla K*v(s)
			\|_{\infty}
		\, \mathrm{d}s \\
		\le \, & 
		C
		\int_0^{t}
			(t-s)^{-\frac{n}{4p}-\frac12}
			\|
				v(s)
			\|_{2p}
			\|
				\nabla K 
			\|_{q_1}
			\|
				v(s)
			\|_{\frac{np}{n-p}}
		\, \mathrm{d}s \\
		\le \, & 
		C
         (\|v_0\|_p+\|v_0\|_p^2)^2
		\int_0^{t}
			(t-s)^{-\frac{n}{4p}-\frac12}
			s^{-\frac{n}{4p}-\frac12}
		\, \mathrm{d}s\\
		\le \, & 
		C
		t^{-\frac{n}{2p}}
		(
			\|
				v_0
			\|_p
			+
			\|
				v_0
			\|_p^2
		)^2
		, \quad t>0, 
\end{aligned}
\end{equation}
where $q_1$ is the exponent given as \eqref{exp;r1} and 
$
	1 < q_1 < {n}/({n-1})
$
for $n/2 < p < n$.
Therefore, decay estimate \eqref{decay;v2} with $q=\infty$ in the case $n/2<p<n$ 
follows from \eqref{no;vr} with $q=\infty$ and \eqref{decay;nIip}. 

It remains to consider the case $p=n$ of \eqref{decay;v2} with $q=\infty$. 
Using \eqref{est;nlrlp}, \eqref{decay;mur} with $q=4n$, Lemma \ref{lem;bessel},
the Cauchy--Schwarz inequality, and Young's inequality, we obtain 
\begin{equation}\label{decay;nIipn}
	\begin{aligned}
		\|
			I(t)
		\|_{\infty}
		\le \, &
		\int_0^{t}
			\left\|
				\nabla 
				\cdot 
				\mathrm{e}^{(t-s)\mathcal{L}_A}
				(
					v
					\nabla K*v
				)(s)
			\right\|_\infty
		\, \mathrm{d}s \\
		\le \, & 
		C
		\int_0^{t}
			(t-s)^{-\frac34}
			\|
				v(s)
				\nabla K*v(s)
			\|_{2n}
		\, \mathrm{d}s \\
		\le \, & 
		C
		\int_0^{t}
			(t-s)^{-\frac34}
			\|
				v(s)
			\|_{4n}
			\|
				\nabla K*v(s)
			\|_{4n}
		\, \mathrm{d}s \\
		\le \, & 
		C
		\int_0^{t}
			(t-s)^{-\frac34}
			\|
				\nabla K
			\|_1
			\|
				v(s)
			\|_{4n}^2
		\, \mathrm{d}s \\
		\le \, & 
		C
		(
			\|v_0\|_{n}
			+
			\|v_0\|_{n}^2
		)^2
		\int_0^{t}
			(t-s)^{-\frac34}
			s^{-\frac34}
		\, \mathrm{d}s\\
		\le \, & 
		C
		t^{-\frac12}
		(
			\|v_0\|_{n}
			+
			\|v_0\|_{n}^2
		)^2
		, \qquad t>0.
\end{aligned}
\end{equation}
Thus, by \eqref{no;vr} with $q=\infty$ and \eqref{decay;nIipn}, we have decay estimate \eqref{decay;v2} with $q=\infty$
in the case $p=n$. 
As a consequence, 
applying the $L^p$-bound in \eqref{dEcay;v0}, \eqref{decay;v2} with $q=\infty$ and 
$
	\|
		f
	\|_q
	\le 
	\|
		f
	\|_p^{p/q}
	\|
		f
	\|_\infty^{1-p/q}
$ 
for 
$p<q<\infty$, we find that the desired estimate \eqref{decay;v2} holds
for all $p \le q \le \infty$, 
and the proof of Lemma \ref{lem;idecayv2} is complete.
\end{proof}

\begin{proof}[{\bf Proof of Theorem \ref{t1}}]
As stated in the introduction, it is enough to show the following decay estimates: 
\begin{align}
\label{decay;va} 
	\sup_{t>0}
		t^{\frac{n}2(\frac1p-\frac1q)}
		\|
			v(t)
		\|_q
	\le
	C
	\|
		v_0
	\|_p,\\
\label{asy;va} 
		\sup_{t>0}
		t^{\frac{n}2(\frac1p-\frac1q)+\frac{n}2(\frac1p-\frac1n)}
		\|
			v(t)-\mathrm{e}^{t\mathcal{L}_A}v_0
		\|_q
		\le
		C'
		\|
			v_0
		\|_p
\end{align}
for all $p \le q \le \infty$, 
where $C, C' > 0$ are independent of $v_0$. 
By Lemmas \ref{lem;idecayv1} and \ref{lem;idecayv2},  \eqref{decay;va} holds.
It remains to prove \eqref{asy;va}. 
By \eqref{eq;ieqv},  
\begin{equation}\label{eq;I}
	v(t)
	-
	\mathrm{e}^{t\mathcal{L}_A}v_0
	=
	-
	I_1(t)
	-
	I_2(t),
\end{equation} 
where 
\begin{align}
	\label{eq;I1}
	I_1(t) 
	\coloneqq \, &
	\int_{0}^{\frac{t}{2}}
		\nabla
		\cdot 
		\mathrm{e}^{(t-s)\mathcal{L}_A}(v\nabla K*v)(s)
	\, \mathrm{d}s, \\
	\label{eq;I2}
	I_2(t)
	\coloneqq \, &
	\int_{\frac{t}{2}}^t
		\nabla
		\cdot 
		\mathrm{e}^{(t-s)\mathcal{L}_A}(v\nabla K*v)(s)
	\, \mathrm{d}s.
\end{align}
Noting that $1-n/(2p)>0$, 
we observe from \eqref{est;nlrlp}, \eqref{decay;va}, Lemma \ref{lem;bessel} and Young's inequality that 
\begin{equation}\label{decay;I1}
	\begin{aligned}
		\|
			I_1(t)
		\|_q
		\le \, &
		C
		\int_0^{\frac{t}{2}}
			(t-s)^{-\frac{n}2(\frac1p-\frac1q)-\frac12}
			\|
				v(s)
				\nabla K*v(s)
			\|_p
		\, \mathrm{d}s \\
		\le \, &
		C
		t^{-\frac{n}2(\frac1p-\frac1q)-\frac12}
		\int_0^{\frac{t}{2}}
			\|
				v(s)
			\|_p
			\|
				\nabla 
				K*v(s)
			\|_\infty
		\, \mathrm{d}s \\
		\le \, &
		C
		t^{-\frac{n}2(\frac1p-\frac1q)-\frac12}
		\int_0^{\frac{t}{2}}
			\|
				v(s)
			\|_p
			\|
				\nabla K
			\|_1
			\|
				v(s)
			\|_\infty
		\, \mathrm{d}s \\
		\le \, &
		C
		t^{-\frac{n}2(\frac1p-\frac1q)-\frac12}
            \|
				v_0
			\|_p^2
		\int_0^{\frac{t}{2}}
			s^{-\frac{n}{2p}}
		\, \mathrm{d}s \\
		\le \, &
		C
		t^{-\frac{n}2(\frac1p-\frac1q)-\frac{n}2(\frac1p-\frac1n)}
		\|
			v_0
		\|_p^2, \qquad t>0.
	\end{aligned}
\end{equation}
Similarly, 
\begin{equation}\label{decay;I2}
	\begin{aligned}
		\|
			I_2(t)
		\|_q
		\le \, &
		C
		\int_{\frac{t}{2}}^t
			(t-s)^{-\frac12}
			\|
				v(s)
				\nabla 
				K*v(s)
			\|_q
		\, \mathrm{d}s \\
		\le \, &
		C
		\int_{\frac{t}{2}}^t
			(t-s)^{-\frac12}
			\|
				v(s)
			\|_q
			\|
				\nabla 
				K
				*
				v(s)
			\|_\infty
		\, \mathrm{d}s \\
		\le \, &
		C
		\int_{\frac{t}{2}}^t
			(t-s)^{-\frac12}
			\|
				v(s)
			\|_q
			\|
				\nabla K
			\|_1
			\|
					v(s)
			\|_\infty
			\, \mathrm{d}s \\
		\le \, &
		C
			\|
				v_0
			\|_p^2
		\int_{t/2}^t
			(t-s)^{-\frac12}
			s^{-\frac{n}2(\frac1p-\frac1q)-\frac{n}{2p}}
		\, \mathrm{d}s \\
		\le \, &
		C
		t^{-\frac{n}2(\frac1p-\frac1q)-\frac{n}2(\frac1p-\frac1n)}
		\|
				v_0
		\|_p^2, \qquad t>0.
\end{aligned}
\end{equation}
Combining the estimates \eqref{decay;I1} and \eqref{decay;I2}, together with the relation
$
	v(t)
	-
	\mathrm{e}^{t\mathcal{L}_A}v_0
	=
	-
	I_1(t)
	-
	I_2(t),
$
yields the desired estimate \eqref{asy;va}. This completes the proof of Theorem \ref{t1}.
\end{proof}

\section{Proof of Theorem \ref{t2}}
Let $A_{\ast}$ be defined by \eqref{const;as}.
Then we recall that the analytic semigroup $\mathrm{e}^{t\mathcal{L}_A}f$ is written as  
\begin{equation}\label{semi;La}
	\mathrm{e}^{t\mathcal{L}_A}
	f
	=
	\mathcal{F}^{-1}
	\big[
		\mathrm{e}^{th_A}\mathcal{F}[f]
	\big], \qquad f\in \mathcal{S}.
\end{equation}
By \cite[Lemma 6.1]{WY}, 
the function $h_A$ defined by \eqref{fn;ha} has the following estimate:
\begin{lem}[Lemma~6.1 in~\cite{WY}]\label{lem;esha0}
Assume 
$
	0
	<
	A
	< 
	A_\ast
$. 
Then we have 
\begin{equation}\label{es;ha0}
	h_A(\xi)
	\le 
	-c_{\ast\ast}|\xi|^2
\end{equation}
for all $\xi \in \mathbb{R}^n$, 
where $c_{\ast\ast}$ is a positive constant defined by 
\begin{equation}\label{const;cs}
	c_{\ast\ast}
	=
	\left\{
		\begin{aligned}
			&1, 
			&&\qquad
			\dfrac{\beta_1}{\beta_2} \le 1, \quad
			\dfrac{\beta_1}{\beta_2} \le \dfrac{\lambda_1}{\lambda_2},\\
			&
			\dfrac{\beta_1\lambda_2-\beta_2\lambda_1}{\lambda_1\lambda_2} 
			\left( 
				\dfrac{\lambda_1\lambda_2}{\beta_1\lambda_2-\beta_2\lambda_1}
				-
				A
			\right), 
			&& \qquad 
			\dfrac{\beta_1}{\beta_2} > \dfrac{\lambda_1}{\lambda_2}, \quad
			\dfrac{\beta_1}{\beta_2} \ge \left( \dfrac{\lambda_1}{\lambda_2} \right)^2,\\
			&
			\dfrac{(\sqrt{\beta_1}-\sqrt{\beta_2})^2}{\lambda_1-\lambda_2}
				\left(
					\dfrac{\lambda_1-\lambda_2}{(\sqrt{\beta_1}-\sqrt{\beta_2})^2}
					-
					A
				\right), 
			&&\qquad 
			\dfrac{\beta_1}{\beta_2} < \left(\dfrac{\lambda_1}{\lambda_2}\right)^2, \quad
			\dfrac{\beta_1}{\beta_2} > 1.
		\end{aligned}
	\right.
\end{equation}
\end{lem}
We recall 
\[
	c_{\ast}
	\coloneqq 
	1
	-
	\frac{\beta_1\lambda_2-\beta_2\lambda_1}{\lambda_1\lambda_2}
	A
	>
	0
\]
when $0 < A < A_*$.
Then, using
$
	(2\pi)^{-n/2}
	\mathcal{F}^{-1}
	\big[
		\mathrm{e}^{-c_{\ast} t|\cdot|^2}
	\big](\cdot)
	=
	G(c_{\ast} t,\cdot)
$ 
for
$t>0$, 
we can rewrite $\mathrm{e}^{t\mathcal{L}_A}f$ as follows:
\begin{equation}
	\begin{aligned}
	\mathrm{e}^{t\mathcal{L}_A}f
	= \, &
	\mathcal{F}^{-1}
	\big[
		\mathrm{e}^{-c_{\ast}t|\cdot|^2}\mathcal{F}[f]
	\big]
	+
	\mathcal{F}^{-1}
	\big[
		(\mathrm{e}^{th_A}-\mathrm{e}^{-c_{\ast}t|\cdot|^2})
		\mathcal{F}[f]
	\big] \\
	= \, &
	\mathrm{e}^{c_{\ast}t\Delta}f+\mathcal{K}_A(t)*f, 
\end{aligned}
\end{equation}
where $\mathrm{e}^{c_{\ast}t\Delta}$ is the standard heat semigroup, and 
\begin{equation}\label{fn;hka}
	\mathcal{K}_A(t,x)
	\coloneqq 
	(2\pi)^{-\frac{n}{2}}
	\mathcal{F}^{-1}
	\big[
		\mathrm{e}^{th_A}-\mathrm{e}^{-c_{\ast}t|\cdot|^2}
	\big](x).
\end{equation}
In what follows, 
we first derive pointwise estimates  
for the Fourier multiplier in \eqref{fn;hka}. 

\subsection{Pointwise estimates in Fourier space}
\begin{lem}\label{lem;esha1}
Let $0<A<A_\ast$. 
Then there exists 
$
	C
	=
	C(n,\beta_1,\beta_2,\lambda_1,\lambda_2,A)
	>
	0
$
such that
\begin{equation}\label{es;hacss}
	|
		\mathcal{F}[\mathcal{K}_A(t)](\xi)
	|
	\le 
	C
	t
	|\xi|^4
	\mathrm{e}^{-c_{\ast\ast} t|\xi|^2}
\end{equation}
holds for all $t>0$ and all $\xi \in \mathbb{R}^n$,
where 
$c_{\ast\ast}>0$ is defined by \eqref{const;cs}, 
and $\mathcal{K}_A(t,\cdot)$ is defined by \eqref{fn;hka}. 
\end{lem}

\begin{rem}
Let $A_\ast$ be defined by \eqref{const;as}, and let $c_{\ast}, c_{\ast\ast} > 0 $ be defined by \eqref{const;ca} and \eqref{const;cs}, respectively. 
If 
$
	0
	<
	A
	<
	A_\ast
$, 
then 
\begin{equation}\label{est;cscss}
	c_{\ast\ast}
	\le 
	c_{\ast}. 
\end{equation}
Indeed,
\[
	c_{\ast}
	-
	c_{\ast\ast}
	=
	-\frac{\beta_1\lambda_2-\beta_2\lambda_1}{\lambda_1\lambda_2}A
	=
	\frac{\beta_2}{\lambda_1}\left(\frac{\lambda_1}{\lambda_2}
	-
	\frac{\beta_1}{\beta_2}\right)A
	\ge 0
\]
if $\beta_1/\beta_2\le 1$, $\beta_1/\beta_2\le \lambda_1/\lambda_2$. 
Also, we obtain $c_{\ast}=c_{\ast\ast}$ if $\beta_1/\beta_2>\lambda_1/\lambda_2$, $\beta_1/\beta_2\ge (\lambda_1/\lambda_2)^2$. 
In addition, note that 
\[
	\dfrac{(\sqrt{\beta_1}-\sqrt{\beta_2})^2}{\lambda_1-\lambda_2}
	>
	\frac{\beta_1\lambda_2-\beta_2\lambda_1}{\lambda_1\lambda_2}
\]
when $\beta_1/\beta_2<(\lambda_1/\lambda_2)^2$, $\beta_1/\beta_2>1$. Then we have 
\[
	c_{\ast}-c_{\ast\ast}
	=
	\left[
		\dfrac{(\sqrt{\beta_1}-\sqrt{\beta_2})^2}{\lambda_1-\lambda_2}
		-
		\frac{\beta_1\lambda_2-\beta_2\lambda_1}{\lambda_1\lambda_2}
	\right]
	A
	>
	0.
\]
Hence \eqref{est;cscss} follows.
\end{rem}

\begin{proof}[{\bf Proof of Lemma \ref{lem;esha1}}]
The function $h_A(\xi)$ given by \eqref{fn;ha} is represented as
\begin{equation}\label{fn;ha1}
	h_A(\xi)
	=
	-c_{\ast}
	|\xi|^2
	+
	H_A(\xi), \qquad \xi \in \mathbb{R}^n,
\end{equation}
where the constant $c_{\ast}>0$ is defined by \eqref{const;ca}, and 
\begin{equation}
	\label{fn;Ha}
	H_A(\xi)
	\coloneqq 
	A
	\left[
		\frac{\beta_2}{\lambda_2(\lambda_2+|\xi|^2)}
		-
		\frac{\beta_1}{\lambda_1(\lambda_1+|\xi|^2)}
	\right]
	|\xi|^4, \qquad \xi \in \mathbb{R}^n.
\end{equation}
By the mean value theorem, \eqref{es;ha0}, \eqref{est;cscss}, \eqref{fn;ha1} and 
$
	|
		H_A(\xi)
	|
	\le 
	C
	|\xi|^4
$, we have 
\begin{align*}
	|\mathcal{F}[\mathcal{K}_A(t)](\xi)|
	\le \, &
	C
	t
	\big|
		h_A(\xi)
		+
		c_{\ast}
		|\xi|^2
	\big|
	\int_0^1
		\mathrm{e}^{-c_{\ast}t|\xi|^2+\theta t(h_A(\xi)+c_{\ast}|\xi|^2)}
	\, \mathrm{d}\theta \\
	\le \, &
	C
	t
	|H_A(\xi)|
	\int_0^1
		\mathrm{e}^{-t(\theta c_{\ast\ast} +(1-\theta)c_{\ast})|\xi|^2}
	\, \mathrm{d}\theta\\
	\le \, &
	C
	t
	|\xi|^4\mathrm{e}^{-c_{\ast\ast} t|\xi|^2}
\end{align*}
for all $t>0$ and all $\xi \in \mathbb{R}^n$. 
Hence the proof of Lemma \ref{lem;esha1} is complete.
\end{proof}

\begin{lem}\label{lem;key}
Let $\mathcal{K}_{A}(t,x)$ be the function defined by \eqref{fn;hka} and $\alpha \in \mathbb{Z}_{+}^n$ satisfy $|\alpha|\ge 1$.
Assume $0<A<A_\ast$. 
Then there exists $C=C(n,\beta_1,\beta_2,\lambda_1,\lambda_2, A, \alpha) > 0$ such that,
for every $t \ge 1$ and $\xi\in \mathbb{R}^n \setminus \{ 0 \}$,
\begin{equation}\label{est;key}
\begin{aligned}
	\, &
	\big|
		\partial_\xi^{\alpha}(\mathcal{F}[\mathcal{K}_A(t)](\xi))
	\big|\\
	\le \, &
	C
	\mathrm{e}^{-c_{\ast\ast} t|\xi|^2}
	\left[
		\sum_{\frac{|\alpha|}{2}\le l\le |\alpha|}
			t^{l+1}
			|\xi|^{2l+4-|\alpha|}
		+
		\sum_{\substack{\alpha_1+\alpha_2=\alpha, \\ |\alpha_2|\ge 1}}
		\sum_{\substack{\frac{|\alpha_1|}{2} \le l\le |\alpha_1|, \\ 1\le m\le |\alpha_2|}}
			t^{l+m}
			|\xi|^{2l+4m-|\alpha|}
	\right]
\end{aligned}
\end{equation}
holds,
where 
$c_{\ast\ast}>0$
is given by \eqref{const;cs}. 
\end{lem}

The following three lemmas are needed to show Lemma \ref{lem;key}. 

\begin{lem}\label{lem;xecaa}
Let $\alpha \in \mathbb{Z}_{+}^n$. 
Then there exists 
$
	C
	=
	C(n, \beta_1, \beta_2, \lambda_1, \lambda_2, A, \alpha)
	>
	0
$ such that 
\begin{equation}
\label{est;ecaa}
	\big|
		\partial_\xi^\alpha 
		\mathrm{e}^{-c_{\ast}t|\xi|^2}
	\big|
	\le 
	C
	\mathrm{e}^{-c_{\ast} t|\xi|^2}
	\sum_{\frac{|\alpha|}{2}\le l\le |\alpha|}
		t^l|\xi|^{2l-|\alpha|}, 
	\quad t>0, \quad \xi \in \mathbb{R}^n \setminus \{ 0 \},
\end{equation}
where $c_{\ast} > 0$ is given by \eqref{const;ca}.
\end{lem}

\begin{proof}
\eqref{est;ecaa} will be proved only in the case $|\alpha| \ge 1$.
Using \cite[Lemma 3.3]{WY} and 
the estimate 
$
	|
		\partial_\xi^\alpha
		|\xi|^2
	|
	\le 
	2
	|\xi|^{2-|\alpha|}
$ 
for every $\alpha \in \mathbb{Z}_{+}^n$ with $|\alpha| \le 2$ and every $\xi \in \mathbb{R}^n$, we obtain  
\begin{align*}
	\big|
		\partial_\xi^\alpha 
		\mathrm{e}^{-c_{\ast} t|\xi|^2}
	\big|
	\le \, &
	\mathrm{e}^{-c_{\ast} t|\xi|^2}
	\sum_{l=1}^{|\alpha|}
		(c_{\ast}t)^l
		\sum_{\substack{\alpha_1+\cdots+\alpha_l=\alpha, \\ |\alpha_i|\ge 1}}
			|
				\Gamma_{\alpha_1,\ldots, \alpha_l}^l
			|
			|
				(\partial_\xi^{\alpha_1}|\xi|^2)
				\cdots
				(\partial_\xi^{\alpha_l}|\xi|^2)
			| \\
	\le \, &
	C
	\mathrm{e}^{-c_{\ast} t|\xi|^2}
	\sum_{\frac{|\alpha|}{2}\le l\le |\alpha|}
		t^l|\xi|^{2l-|\alpha|}, \qquad  
		t>0, \quad \xi \in \mathbb{R}^n \setminus \{ 0 \},
\end{align*}
where 
\[
	C
	\coloneqq 
	\max_{ \frac{|\alpha|}{2}\le l\le |\alpha| }
	\left[
		(2c_{\ast})^l
		\sum_{
			\substack{
				\alpha_1+\cdots+\alpha_l=\alpha, \\ 
				|\alpha_i|\ge 1
			}
		}
		|
			\Gamma_{\alpha_1,\ldots, \alpha_l}^l
		|
	\right]
	>
	0.
\]
Thus the proof of Lemma \ref{lem;xecaa} is complete.
\end{proof}

\begin{lem}\label{lem;HAest}
Let $\alpha \in \mathbb{Z}_{+}^n$. 
Then there exists a constant 
$
	C
	=
	C(n,\beta_1,\beta_2,\lambda_1, \lambda_2,  A, \alpha)
	>0
$ 
such that
\begin{equation}\label{est;HA}
	\big|
		\partial_\xi^\alpha H_A(\xi)
	\big|
	\le 
	C
	|\xi|^{4-|\alpha|}
	, \qquad 
	\xi \in \mathbb{R}^n \setminus \{ 0 \}, 
\end{equation}
where $H_A(\xi)$ is defined by \eqref{fn;Ha}.
\end{lem}

\begin{proof}
We will only prove \eqref{est;HA} in the case $|\alpha|\ge 1$. 
By direct calculation, we observe
$
	|
		\partial_\xi^\alpha|\xi|^4
	|
	\le 
	C
	|\xi|^{4-|\alpha|}
$
for any $\alpha \in \mathbb{Z}_{+}^n$ with $|\alpha|\le 4$ and all $\xi \in \mathbb{R}^n$. 
It also follows from \cite[Lemma 3.2]{WY} that 
\[
	\left|
		\partial_\xi^{\alpha}
		\left[
			\frac{\beta_2}{\lambda_2(\lambda_2+|\xi|^2)}
			-
			\frac{\beta_1}{\lambda_1(\lambda_1+|\xi|^2)}
		\right]
	\right|
	\le 
	C
	(1+|\xi|)^{-2-|\alpha|}
\]
for $\alpha \in \mathbb{Z}_{+}^n$, $\xi \in \mathbb{R}^n$ and $C = C(n, \beta_1,\beta_2,\lambda_1, \lambda_2, A, \alpha) > 0$. 
Together with Leibniz's rule and the identity
$
	\partial_\xi^\alpha
		|
			\xi
		|^4
		=
		0
$ 
$
	(
		|\alpha|
		\ge 
		5
	)
$, these estimates imply that  
\begin{align*}
	\, & 
	|\xi|^{|\alpha|-4}
	\big|
		\partial_\xi^\alpha H_A(\xi)
	\big| \\
	= \, &
	|\xi|^{|\alpha|-4}
	\left|
		\sum_{\substack{\alpha_1+\alpha_2=\alpha \\ |\alpha_1|\le 4}}
		C( A, \alpha, \alpha_1, \alpha_2)
		\partial_\xi^{\alpha_1}|\xi|^4
		\partial_\xi^{\alpha_2}
		\left[
			\frac{\beta_2}{\lambda_2(\lambda_2+|\xi|^2)}
			-
			\frac{\beta_1}{\lambda_1(\lambda_1+|\xi|^2)}
		\right]
	\right| \\
	\le \, &
	\sum_{\substack{\alpha_1+\alpha_2=\alpha \\ |\alpha_1|\le 4}}
		C(n, \beta_1, \beta_2, \lambda_1, \lambda_2, A,\alpha, \alpha_1, \alpha_2)
		|\xi|^{|\alpha|-|\alpha_1|}(1+|\xi|)^{-2-|\alpha_2|} \\
	\le \, & 
	\sum_{\substack{\alpha_1+\alpha_2=\alpha \\ |\alpha_1|\le 4}}
		C(n, \beta_1, \beta_2, \lambda_1, \lambda_2, A,\alpha, \alpha_1, \alpha_2)
\end{align*}
for any $\alpha\in \mathbb{Z}_{+}^n$ with $|\alpha| \ge 1$ and all $\xi \in \mathbb{R}^n \setminus \{ 0 \}$. 
Therefore Lemma \ref{lem;HAest} holds. 
\end{proof}

\begin{lem}
Let $\alpha \in \mathbb{Z}_{+}^n$ satisfy $| \alpha | \ge 1$, and let $H_A$ be the function defined in \eqref{fn;Ha}.
Then there exists 
$
	C
	=
	C(n, \beta_1, \beta_2, \lambda_1, \lambda_2, A, \alpha)
	>
	0
$ 
such that 
\begin{equation}
\label{est;eHa}
	|
		\partial_\xi^\alpha 
		\mathrm{e}^{tH_A(\xi)}
	|
	\le 
	C
	\mathrm{e}^{tH_A(\xi)}
	\sum_{m=1}^{|\alpha|}
		t^m|\xi|^{4m-|\alpha|}, \qquad t>0, \quad \xi \in \mathbb{R}^n \setminus \{ 0 \}.
\end{equation}
\end{lem}

\begin{proof} 
Applying \cite[Lemma 3.3]{WY} and \eqref{est;HA} yields 
\begin{align*}
	&\, 
	|
		\partial_\xi^\alpha \mathrm{e}^{tH_A(\xi)}
	|\\
	\le \, &
	\mathrm{e}^{tH_A(\xi)}
	\sum_{m=1}^{|\alpha|}
		t^m
		\sum_{\substack{\alpha_1+\cdots+\alpha_m=\alpha, \\ |\alpha_i|\ge 1}}
			|
				C(m,\alpha_1,\ldots, \alpha_m)
			|
			\big|
				(
					\partial_\xi^{\alpha_1}H_A(\xi)
				)
				\cdots
				(
					\partial_\xi^{\alpha_m}H_A(\xi)
				)
			\big| \\
	\le \, &
	\mathrm{e}^{tH_A(\xi)}
	\sum_{m=1}^{|\alpha|}
		t^m
		|\xi|^{4m-|\alpha|}
		\left(
			\sum_{\substack{\alpha_1+\cdots+\alpha_m=\alpha, \\ |\alpha_i|\ge 1}}
				C(n, \beta_1, \beta_2, \lambda_1, \lambda_2, A, m, \alpha_1,\ldots, \alpha_m)
		\right)\\
	\le \, &
	C
		\mathrm{e}^{tH_A(\xi)}
		\sum_{m=1}^{|\alpha|}
			t^m|\xi|^{4m-|\alpha|}
\end{align*}
for $t>0$ and $\xi \in \mathbb{R}^n \setminus \{ 0 \}$,
where
\[
	C
	\coloneqq 
	\max_{1\le m\le |\alpha|}
	\left[
		\sum_{\substack{\alpha_1+\cdots+\alpha_m=\alpha, \\ |\alpha_i|\ge 1}}
			C(n, \beta_1, \beta_2, \lambda_1, \lambda_2, A, m, \alpha_1,\ldots, \alpha_m)
	\right]
	>
	0.
\] 
This proves \eqref{est;eHa}.
\end{proof}

\begin{proof}[{\bf Proof of Lemma \ref{lem;key}}]
It follows from \eqref{fn;ha1} and Leibniz's rule that 
for $\alpha \in \mathbb{Z}_{+}^n$ with $|\alpha|\ge 1$, $t \ge 1$ and $\xi \in \mathbb{R}^n \setminus \{ 0 \}$,
\begin{equation}\label{est;key0}
	\begin{aligned}
		\partial_\xi^{\alpha}
		\left(
			\mathcal{F}[\mathcal{K}_A(t)](\xi)
		\right)
		= \, &
		(2\pi)^{-\frac{n}2}
		\partial_\xi^{\alpha}
		[
			\mathrm{e}^{th_A(\xi)}
			-
			\mathrm{e}^{-c_{\ast}t|\xi|^2}
		]\\
		= \, &
		(2\pi)^{-\frac{n}2}
		\sum_{\alpha_1+\alpha_2=\alpha}
			C(\alpha, \alpha_1, \alpha_2)
			(
				\partial_\xi^{\alpha_1}
				\mathrm{e}^{-c_{\ast}t|\xi|^2}
			)
			\partial_\xi^{\alpha_2}
			(
				\mathrm{e}^{tH_A(\xi)}
				-
				1
			) \\
		= \, &
		(2\pi)^{-\frac{n}2}
		(
			\partial_\xi^{\alpha}
			\mathrm{e}^{-c_{\ast}t|\xi|^2}
		)
		(
			\mathrm{e}^{tH_A(\xi)}
			-
			1
		) \\
		& \, +
		(2\pi)^{-\frac{n}2}
		\sum_{\substack{\alpha_1+\alpha_2=\alpha, \\ |\alpha_2|\ge 1}}
			C(\alpha, \alpha_1, \alpha_2)
			(
				\partial_\xi^{\alpha_1}
				\mathrm{e}^{-c_{\ast}t|\xi|^2}
			)
			(
				\partial_\xi^{\alpha_2}\mathrm{e}^{tH_A(\xi)}
			) \\
		\eqqcolon \, &
		J_1(t,\xi)+J_2(t,\xi).
	\end{aligned}
\end{equation}
By \eqref{es;hacss}, \eqref{fn;ha1} and \eqref{est;ecaa}, the estimate of $|J_1(t,\xi)|$ is given by 
\begin{equation}\label{est;key1}
\begin{aligned}
	|J_1(t,\xi)|
	\le \, &
	C
	\sum_{\frac{|\alpha|}{2}\le l\le |\alpha|}
		t^l
		|\xi|^{2l-|\alpha|}
		(2\pi)^{-\frac{n}2}
		\mathrm{e}^{-c_{\ast} t|\xi|^2}
		\big|
			\mathrm{e}^{tH_A(\xi)}
			-
			1
		\big| \\
	= \, &
	C
	\sum_{\frac{|\alpha|}{2}\le l\le |\alpha|}
		t^l
		|\xi|^{2l-|\alpha|}
		|
			\mathcal{F}
			[\mathcal{K}_A(t)](\xi)
		| \\
	\le \, &
	C
	\mathrm{e}^{-c_{\ast\ast} t|\xi|^2}
	\sum_{\frac{|\alpha|}{2}\le l\le |\alpha|}
		t^{l+1}
		|\xi|^{2l+4-|\alpha|} .
\end{aligned}
\end{equation}
We use \eqref{es;ha0}, \eqref{fn;ha1}, \eqref{est;ecaa} and \eqref{est;eHa} to estimate $|J_2(t,\xi)|$ as follows: 
\begin{equation}\label{est;key2}
	\begin{aligned}
		|J_2&(t,\xi)| \\
		\le \, & 
		(2\pi)^{-\frac{n}2}
		\sum_{\substack{\alpha_1+\alpha_2=\alpha, \\ |\alpha_2|\ge 1}}
			C(\alpha, \alpha_1, \alpha_2)
			\big|
				\partial_\xi^{\alpha_1}
				\mathrm{e}^{-c_{\ast}t|\xi|^2}
			\big|
			\big|
				\partial_\xi^{\alpha_2}
				\mathrm{e}^{tH_A(\xi)}
			\big| \\
		\le \, & 
		\mathrm{e}^{th_A(\xi)}
		\sum_{\substack{\alpha_1+\alpha_2=\alpha, \\ |\alpha_2|\ge 1}}
		\sum_{\substack{\frac{|\alpha_1|}{2}\le l\le |\alpha_1|, \\ 1\le m\le |\alpha_2|}}
			C(n, \beta_1, \beta_2, \lambda_1, \lambda_2, A, \alpha, \alpha_1, \alpha_2)
			t^{m+l}|\xi|^{2l+4m-|\alpha|} \\
		\le \, & 
		C
		\mathrm{e}^{-c_{\ast\ast}t|\xi|^2}
		\sum_{\substack{\alpha_1+\alpha_2=\alpha, \\ |\alpha_2|\ge 1}}
		\sum_{\substack{\frac{|\alpha_1|}{2}\le l\le |\alpha_1|, \\ 1\le m\le |\alpha_2|}}
			t^{l+m}
			|\xi|^{2l+4m-|\alpha|}.
	\end{aligned}
\end{equation}
Thus, combining \eqref{est;key1} and \eqref{est;key2} with \eqref{est;key0}, we get the desired estimate \eqref{est;key}.
\end{proof}

\subsection{$L^q$-estimates for $\mathcal{K}_A(t)$}
Using Lemmas \ref{lem;esha1} and \ref{lem;key}, 
we derive $L^q$-estimates for $\mathcal{K}_A(t)$, defined by \eqref{fn;hka}. 
We first consider the cases $q=2$ and $q = \infty$. 
\begin{lem}\label{lem;ifhkaie}
Assume $0<A<A_\ast$. 
Then, there exist constants $C$, $C^{\prime}>0$ such that for every $t>0$,
\begin{align}\label{es;ifhka2e0}
	\|
		\mathcal{K}_A(t)
	\|_2
	\le &
	C
	t^{-\frac{n}4-1},\\
\label{es;ifhkaie0}
	\|
		\mathcal{K}_A(t)
	\|_\infty
	\le &
	C^{\prime}
	t^{-\frac{n}2-1}.
\end{align}
\end{lem}

\begin{proof}
We first prove \eqref{es;ifhka2e0}. 
Applying Parseval's identity, \eqref{es;hacss} and \eqref{es;ede}, we get 
\begin{align*}
	\|
		\mathcal{K}_A(t)
	\|_2^2
	= \, &
	\|
		\mathcal{F}[\mathcal{K}_A(t)]
	\|_2^2\\
	\le \, &
	C
	t^2
	\Big\|
		|\cdot|^{8}\mathrm{e}^{-2tc_{\ast\ast}|\cdot|^2}
	\Big\|_1\\
	\le \, &
	C
	t^{-\frac{n}2-2}
\end{align*}
for $t>0$, 
which implies estimate \eqref{es;ifhka2e0}.

We next prove \eqref{es;ifhkaie0}. 
Using \eqref{es;hacss} and \eqref{es;ede}, we obtain, for every $t > 0$ and $x \in \mathbb{R}^n$,
\begin{equation}\label{est;jka}
	|
		\mathcal{K}_A(t,x)
	|
	\le 
	(2\pi)^{-\frac{n}2}
	\|
		\mathcal{F}
		[
			\mathcal{K}_A(t)
		]
	\|_1
	\le \, C
	t
	\Big\|
		|\cdot|^{4}
		\mathrm{e}^{-c_{\ast\ast} t|\cdot|^2}
	\Big\|_1
	\le 
	C
	t^{-\frac{n}2-1}.
\end{equation}
Hence, \eqref{es;ifhkaie0} follows from \eqref{est;jka}. 
\end{proof}

To obtain the $L^1$-estimate for $K_A(t,\cdot)$, we use the derivative bounds established above.
\begin{lem}\label{lem;keyK}
Suppose that $0<A<A_\ast$. 
Then, the following estimate holds: 
\begin{equation}\label{est;keyK}
	\|
		\mathcal{K}_A(t)
	\|_1 
	\le 
	C
	t^{-1}, \qquad t\ge 1,
\end{equation}
where $C > 0$ is independent of $t$.
\end{lem}

\begin{proof}
Let 
$
	N
	=
	\max
	\big\{
		2,
		\left\lfloor 
			\frac{n}2
		\right\rfloor
		+
		1
	\big\}
$. 
Then we note that $N > \max \{ 1,n/2 \}$ and $4-N> -n/2$.
From \eqref{es;ede} and \eqref{est;key} we see that 
\begin{equation}\label{est;keyK0}
\begin{aligned}
	\|
		\partial_\xi^\alpha
		\mathcal{F}[\mathcal{K}_A(t)]
	\|_2
	\le \, &
	C
	\sum_{\frac{|\alpha|}{2}\le l\le |\alpha|}
		t^{l+1}
		\Big\|
			|\cdot|^{2l+4-|\alpha|}
			\mathrm{e}^{-c_{\ast\ast} t|\cdot|^2}
		\Big\|_2\\
	& \, +
	C
	\sum_{\substack{\alpha_1+\alpha_2=\alpha, \\ |\alpha_2|\ge 1}}
	\sum_{\substack{\frac{|\alpha_1|}{2}\le l\le |\alpha_1|, \\ 1\le m\le |\alpha_2|}}
		t^{l+m}
		\Big\|
			|\cdot|^{2l+4m-|\alpha|}
			\mathrm{e}^{-c_{\ast\ast} t|\cdot|^2}
		\Big\|_2 \\
	\le \, &
	C
	t^{-\frac{n}4+\frac{|\alpha|}2-1}
\end{aligned}
\end{equation}
for $t \ge 1$ and $|\alpha|=N$. 
By \cite[Lemma 3.4]{WY}, 
Lemma \ref{lem;ifhkaie} and \eqref{est;keyK0}, we observe that 
\begin{align*}
	\|
		\mathcal{K}_A(t)
	\|_1
	\le \, &
	C
	\|
		\mathcal{F}[\mathcal{K}_A(t)]
	\|_2^{1-\frac{n}{2N}}
	\left(
		\sum_{|\alpha| = N}
			\|
				\partial_\xi^{\alpha}
				\mathcal{F}[\mathcal{K}_A(t)]
			\|_2
	\right)^{\frac{n}{2N}} \\
	\le \, &
	C
	(
		t^{-\frac{n}4-1}
	)^{1-\frac{n}{2N}}
	(
		t^{-\frac{n}4+\frac{N}2-1}
	)^{\frac{n}{2N}} \\
	\le \, &
	C
	t^{-1}
\end{align*}
for $t \ge 1$, 
where $C>0$ is independent of $t$. 
Hence we conclude that \eqref{est;keyK} holds. 
\end{proof}

\begin{proof}[{\bf Proof of Theorem \ref{t2}}]
Let $0<A<A_\ast$. 
By \eqref{es;ifhkaie0} and \eqref{est;keyK}, we have 
\begin{equation}\label{est;t20}
\begin{aligned}
	\|
		\mathcal{K}_A(t)
	\|_q
	\le \, &
	\|
		\mathcal{K}_A(t)
	\|_1^{\frac1q}
	\|
		\mathcal{K}_A(t)
	\|_\infty^{1-\frac1q} \\
	\le \, &
	C
	(
		t^{-1}
	)^{\frac1q}
	(
		t^{-\frac{n}2-1}
	)^{1-\frac1q} \\
	\le \, &
	C
	t^{-\frac{n}2(1-\frac1q)}
	(1+t)^{-1}
\end{aligned}
\end{equation}
for $t \ge 1$ and $1 \le q \le \infty$. 
Thus the estimate \eqref{asy;la} in the case $t \ge 1$ follows from \eqref{est;t20} and Young's inequality,
that is,
\[
	\| 
		\mathcal{K}_A(t) * f
	\|_q
	\le
	\| 
		\mathcal{K}_A(t)
	\|_{r}
	\| 
		f
	\|_p
	\le
	C
	t^{-\frac{n}2(\frac1p-\frac1q)}
	(1+t)^{-1}
	\|
		f
	\|_p
\]
for 
$1 \le r \le \infty$ so that 
$
	1
	+
	\frac1q
	=
	\frac1p
	+
	\frac1r
$.
Also, by \eqref{est;lrlp} and the $L^p$-$L^q$ estimate for the heat semigroup, 
\begin{equation*}
	\begin{aligned}
		\|
			\mathrm{e}^{t\mathcal{L}_A}f
			-
			\mathrm{e}^{c_*t\Delta}f
		\|_q
		\le \, &
		\|
			\mathrm{e}^{t\mathcal{L}_A}f
		\|_q
		+
		\|
			\mathrm{e}^{c_*t\Delta}f
		\|_q \\
		\le \, &
		Ct^{-\frac{n}{2}(\frac1p-\frac1q)}
		\|
			f
		\|_p
		+
		C't^{-\frac{n}{2}(\frac1p-\frac1q)}
		\|
			f
		\|_p\\
		\le \, &
		C''
		t^{-\frac{n}{2}(\frac1p-\frac1q)}(1+t)^{-1}
		\|
			f
		\|_p
	\end{aligned}
\end{equation*}
holds for $0 < t \le 1$, where $C, C', C''> 0$ are independent of $t$ and $f$. 
Therefore we get the desired estimate \eqref{asy;la}.
\end{proof}

\section{Proof of Theorem \ref{t3}}
In this section, we prove Theorem \ref{t3}. Throughout this section, we assume that $p$, $A$ and $v_0$ satisfy
the same conditions as in Proposition \ref{p1}, and recall
that the decay estimates \eqref{decay;va} and \eqref{asy;va} hold for the global solution $v$ to problem \eqref{eq;Q} obtained in Proposition \ref{prop;qgs} when the initial perturbation $v_0$ in $L^p$-norm is small enough, that is $\|v_0\|_p<\varepsilon_\ast$.
As in Section 3, we assume that the constant $\varepsilon_\ast$ in Proposition \ref{prop;qgs} satisfies $0<\varepsilon_\ast\le 1$.
We treat the subcritical case $n \ge 2$, $n/2<p<n$ and the critical case $n\ge 1$, $p=n$, separately.

\subsection{In the subcritical case $n\ge 2$, $n/2<p<n$}
\begin{prop}\label{prop;t32}
Let $n \ge 2$ and $n/2<p<n$. 
Then we have 
\begin{equation}\label{asy;mt12}
	t^{\frac{n}2(\frac1p-\frac1q)}
	\|
		v(t)-\mathrm{e}^{c_{\ast} t\Delta}v_0
	\|_q
	\le 
	C
	\left[
	t^{-1}
	+
	t^{-\frac{n}2(\frac1p-\frac1n)}
	\right]
	\|
		v_0
	\|_p
\end{equation}
for $t>0$, $p \le q \le \infty$, where $c_{\ast}$ is the positive constant defined in \eqref{const;ca}, 
and $\mathrm{e}^{c_{\ast}t\Delta}$ $(t>0)$ is the heat semigroup. 
\end{prop}

\begin{proof}
By the triangle inequality,
\[
	\|
		v(t)
		-
		\mathrm{e}^{c_{\ast} t\Delta}v_0
	\|_q
	\le
	\|
		v(t)
		-
		\mathrm{e}^{t\mathcal{L}_A}v_0
	\|_q
	+
	\|
		\mathrm{e}^{t\mathcal{L}_A}v_0
		-
		\mathrm{e}^{c_{\ast} t\Delta}v_0
	\|_q,\qquad t > 0.
\]
Combining this inequality with \eqref{asy;la} and \eqref{asy;va}
gives \eqref{asy;mt12}.
\end{proof}

\subsection{In the critical case $n\ge 1$, $p=n$}
Applying the argument of \cite[Theorem 3]{N}, 
we prove the following proposition. 
\begin{prop}\label{prop;t31}
Let $n \ge 1$ and $p=n$.
For every $n \le q \le \infty$, 
there exists 
$
	\varepsilon_0 
	=
	\varepsilon_0(q)
	>0 
$
such that, if $\| v_0 \|_n <\varepsilon_0$, 
then  
\begin{equation}\label{Asy;mt}
	\lim_{t\to\infty}
	t^{\frac{n}2(\frac1n-\frac1q)}
	\|
		v(t)
		-
		\mathrm{e}^{c_{\ast} t\Delta}v_0
	\|_q
	=
	0,
\end{equation} 
where $c_{\ast}$ is the positive constant defined in \eqref{const;ca}, 
and $\mathrm{e}^{c_{\ast}t\Delta}$ $(t>0)$ is the heat semigroup. 
\end{prop}

To prove Proposition \ref{prop;t31}, 
we rewrite \eqref{eq;ieqv} in the following form.
For $t>0$, we have 
\begin{equation}\label{r;v}
	v(t)
	-
	\mathrm{e}^{c_{\ast}t\Delta}v_0
	=
	\sum_{j=1}^4
		V_j(t), 
\end{equation}
where 
\begin{align}
	\label{eq;v1}
	V_1(t)
	\coloneqq \, &
	\mathrm{e}^{t\mathcal{L}_A}v_0
	-
	\mathrm{e}^{c_{\ast}t\Delta}v_0, \\
	\label{eq;v2}
	V_2(t)
	\coloneqq \, &
	-
	\int_0^t
		\nabla 
		\cdot 
		\mathrm{e}^{(t-s)\mathcal{L}_A}
		\Big[
			\mathrm{e}^{c_{\ast}s\Delta}v_0 
			\nabla 
			K*(\mathrm{e}^{c_{\ast}s\Delta}v_0)
		\Big]
	\, \mathrm{d}s, \\
	\label{eq;v3}
	V_3(t)
	\coloneqq \, &
	-
	\int_0^t
		\nabla 
		\cdot 
		\mathrm{e}^{(t-s)\mathcal{L}_A}
		\Big[
			(
				v(s)
				-
				\mathrm{e}^{c_{\ast}s\Delta}v_0
			)
			\nabla 
			K*v(s)
		\Big]
	\, \mathrm{d}s, \\
	\label{eq;v4}
	V_4(t)
	\coloneqq \, &
	-
	\int_0^t
		\nabla 
		\cdot 
		\mathrm{e}^{(t-s)\mathcal{L}_A}
		\Big[
			\mathrm{e}^{c_{\ast}s\Delta}v_0
			\nabla K*(v(s)-\mathrm{e}^{c_{\ast}s\Delta}v_0)
		\Big]
	\, \mathrm{d}s.
\end{align}
We next estimate $V_2(t)$, defined in \eqref{eq;v2}, in $L^q(\mathbb{R}^n)$.

\begin{lem}
Let $n \le q \le \infty$ and $0 < \sigma < 1/2$. 
Then there exists $C > 0$ independent of $t$ and $v_0$ such that 
\begin{equation}\label{Decay;v2}
	\sup_{t > 0}
	t^{\frac{n}2(\frac1n-\frac1q)}
	(1+t)^{\sigma}
	\|
		V_2(t)
	\|_q
	\le 
	C
	\|
		v_0
	\|_n.
\end{equation}
\end{lem}

\begin{proof}
Let $0 < t \le 2$. 
Then we use \eqref{est;nlrlp}, 
$L^p$-$L^q$ estimates for the heat semigroup, 
Lemma \ref{lem;bessel}
and Young's inequality to get 
\begin{equation}\label{est;v2tl}
	\begin{aligned}
		\|
			V_2(t)
		\|_q
		\le \, & 
		C
		\int_0^t
			(t-s)^{-\frac{n}2(\frac{n+q}{2nq}-\frac1q)-\frac12}
			\|
				\mathrm{e}^{c_{\ast}s\Delta}v_0
				\nabla K*(\mathrm{e}^{c_{\ast}s\Delta}v_0)
			\|_{\frac{2nq}{n+q}}
		\, \mathrm{d}s \\
		\le \, & 
		C
		\int_0^t
			(t-s)^{-\frac34+\frac{n}{4q}}
			\|
				\mathrm{e}^{c_{\ast}s\Delta}v_0
			\|_\infty
			\|
				\nabla K
			\|_1
			\|
				\mathrm{e}^{c_{\ast}s\Delta}v_0
			\|_{\frac{2nq}{n+q}}
		\, \mathrm{d}s \\
		\le \, & 
		C
		\|
			v_0
		\|_n^2
		\int_0^t
			(t-s)^{-\frac34+\frac{n}{4q}}
			s^{-\frac34+\frac{n}{4q}}
		\, \mathrm{d}s\\
		\le \, & 
		C
		t^{-\frac{n}2(\frac1n-\frac1q)}(1+t)^{-\sigma}
		\|
			v_0
		\|_n^2
		.
\end{aligned}
\end{equation}
On the other hand, let $t\ge 2$. 
Then we split the function $V_2(t)$ into three parts: 
\begin{equation}\label{eq;v2r}
	V_2(t)
	=
	-
	V_{21}(t)
	-
	V_{22}(t)
	-
	V_{23}(t),
\end{equation}
where 
\begin{align*}
	V_{21}(t)
	\coloneqq \, &
	\int_0^1
		\nabla 
		\cdot 
		\mathrm{e}^{(t-s)\mathcal{L}_A}
		\left[
			\mathrm{e}^{c_{\ast}s\Delta}v_0
			\nabla K*(\mathrm{e}^{c_{\ast}s\Delta}v_0)
		\right]
	\, \mathrm{d}s, \\
	V_{22}(t)
	\coloneqq \, &
	\int_1^{\frac{t}{2}}
		\nabla 
		\cdot 
		\mathrm{e}^{(t-s)\mathcal{L}_A}
		\left[
			\mathrm{e}^{c_{\ast}s\Delta}v_0
			\nabla K*(\mathrm{e}^{c_{\ast}s\Delta}v_0)
		\right]
	\, \mathrm{d}s, \\
	V_{23}(t)
	\coloneqq \, &
	\int_{\frac{t}{2}}^t
		\nabla 
		\cdot 
		\mathrm{e}^{(t-s)\mathcal{L}_A}
		\left[
			\mathrm{e}^{c_{\ast}s\Delta}v_0
			\nabla K*(\mathrm{e}^{c_{\ast}s\Delta}v_0)
		\right]
	\, \mathrm{d}s.
\end{align*}
Using \eqref{est;nlrlp}, $L^p$-$L^q$ estimates for the heat semigroup, 
Lemma \ref{lem;bessel} and Young's inequality again, we have  
\begin{equation}\label{est;v21th}
\begin{aligned}
	\|V_{21}(t)\|_q 
	\le \, & 
	C
	\int_0^1
		(t-s)^{-\frac{n}2(\frac1n-\frac1q)-\frac12}
		\|
			\mathrm{e}^{c_{\ast}s\Delta}v_0
			\nabla K*(\mathrm{e}^{c_{\ast}s\Delta}v_0)
		\|_n\, 
	\mathrm{d}s \\
	\le \, & 
	C
	(t-1)^{-\frac{n}2(\frac1n-\frac1q)-\frac12}
	\int_0^1
		\|
			\mathrm{e}^{c_{\ast}s\Delta}v_0
		\|_\infty
		\|
			\nabla K
		\|_1
		\|
			\mathrm{e}^{c_{\ast}s\Delta}v_0
		\|_n
	\, \mathrm{d}s \\
	\le \, & 
	C
	t^{-\frac{n}2(\frac1n-\frac1q)-\frac12}
	\|
		v_0
	\|_n^2
	\int_0^1
		s^{-\frac12}
	\, \mathrm{d}s \\
	\le \, & 
	C
	t^{-\frac{n}2(\frac1n-\frac1q)}(1+t)^{-\sigma}
	\|
		v_0
	\|_n^2.
\end{aligned}
\end{equation} 
Also, noting 
$
	\nabla K*(\mathrm{e}^{c_{\ast}s\Delta}v_0)
	=
	K*(\nabla \mathrm{e}^{c_{\ast}s\Delta}v_0)
$
and
arguing as in the proofs of \eqref{est;v2tl} and \eqref{est;v21th},
we obtain 
\begin{equation}\label{est;v22th}
\begin{aligned}
	\|
		V_{22}(t)
	\|_q 
	\le \, & 
	C
	\int_1^{\frac{t}{2}}
		(t-s)^{-\frac{n}2(\frac1n-\frac1q)-\frac12}
		\|
			\mathrm{e}^{c_{\ast}s\Delta}v_0 
			K*(\nabla\mathrm{e}^{c_{\ast}s\Delta}v_0)
		\|_n
	\, \mathrm{d}s \\
	\le \, & 
	C
	t^{-\frac{n}2(\frac1n-\frac1q)-\frac12}
	\int_1^{\frac{t}{2}}
		\|
			\mathrm{e}^{c_{\ast}s\Delta}v_0
		\|_\infty
		\|
			K
		\|_1
		\|
			\nabla
			\mathrm{e}^{c_{\ast}s\Delta}v_0
		\|_n
	\, \mathrm{d}s \\
	\le \, & 
	C
	t^{-\frac{n}2(\frac1n-\frac1q)-\frac12}
	\|
		v_0
	\|_n^2
	\int_1^{\frac{t}{2}}
		s^{-1}
	\, \mathrm{d}s \\
	\le \, & 
	C
	t^{-\frac{n}2(\frac1n-\frac1q)}
	t^{-\sigma}
	\|
		v_0
	\|_n^2
	\int_1^{\infty}
		s^{-\frac32+\sigma}
	\, \mathrm{d}s \\
	\le \, & 
	C
	t^{-\frac{n}2(\frac1n-\frac1q)}
	(1+t)^{-\sigma}
	\|
		v_0
	\|_n^2
	,
\end{aligned}
\end{equation}
and 
\begin{equation}\label{est;v23th}
\begin{aligned}
	\|
		V_{23}(t)
	\|_q 
	\le \, & 
	C
	\int_{\frac{t}{2}}^{t}
		(t-s)^{-\frac{n}{2}(\frac{n+q}{2nq}-\frac1q)-\frac12}
		\|
			\mathrm{e}^{c_{\ast}s\Delta}v_0 
			K*(\nabla\mathrm{e}^{c_{\ast}s\Delta}v_0)
		\|_{\frac{2nq}{n+q}}
	\, \mathrm{d}s \\
	\le \, & 
	C
	\int_{\frac{t}{2}}^{t}
		(t-s)^{-\frac34+\frac{n}{4q}}
		\|
			\mathrm{e}^{c_{\ast}s\Delta}v_0
		\|_\infty
		\|
			K
		\|_1
		\|
			\nabla\mathrm{e}^{c_{\ast}s\Delta}v_0
		\|_{\frac{2nq}{n+q}}
	\, \mathrm{d}s \\
	\le \, & 
	C
	\|
		v_0
	\|_n^2
	\int_{\frac{t}{2}}^t
		(t-s)^{-\frac34+\frac{n}{4q}}
		s^{-\frac54+\frac{n}{4q}}
	\, \mathrm{d}s \\
	\le \, & 
	C
	t^{-\frac{n}2(\frac1n-\frac1q)}
	(1+t)^{-\sigma}
	\|
		v_0
	\|_n^2.
\end{aligned}
\end{equation}
Hence, combining \eqref{eq;v2r}, \eqref{est;v21th}, \eqref{est;v22th} and \eqref{est;v23th} leads to  
\[
	t^{\frac{n}2(\frac1n-\frac1q)}
	(1+t)^{\sigma}
	\|
		V_{2}(t)
	\|_q
	\le 
	C
	\|
		v_0
	\|_n^2
\]
for $t\ge 2$, 
which together with \eqref{est;v2tl} implies \eqref{Decay;v2}. 
\end{proof}

It remains to derive $L^q$-estimates for $V_3(t)$ and $V_4(t)$ defined by \eqref{eq;v3} and \eqref{eq;v4}, 
respectively. 
\begin{lem}
Assume $n \le q \le \infty$ and $0 < \sigma < 1/4$. 
Then there exists $C>0$ independent of $t$ and $v_0$
such that  
\begin{equation}\label{Decay;v3v4}
	t^{\frac{n}2(\frac1n-\frac1q)}
	(1+t)^{\sigma}
	(
		\|
			V_3(t)
		\|_q
		+
		\|
			V_4(t)
		\|_q
	)
	\le 
	C
	F_q(t)
	\|
		v_0
	\|_n
, \quad t>0, 
\end{equation}
where 
\begin{equation}\label{eq;Fq}
	F_q(t)
	\coloneqq 
	\sup_{0<\tau\le t}
		\tau^{\frac{n}2(\frac1n-\frac1q)}
		(1+\tau)^\sigma
		\|
			v(\tau)
			-
			\mathrm{e}^{c_{\ast}\tau\Delta}v_0
		\|_q
	.
\end{equation}
\end{lem}

\begin{proof}
We estimate only $V_3(t)$, since $ V_4(t) $ can be treated in the same way. 
Note that 
\[
	\frac{1+t}{1+s}
	\le 
	\frac{t}{s}
\]
for $0 < s \le t$. 
By \eqref{est;nlrlp}, \eqref{decay;va}, Lemma \ref{lem;bessel}, 
H\"{o}lder's and Young's inequalities, 
we have
\begin{align*}
	\|
		V_3(t)
	\|_q
	\le \, & 
	C
	\int_0^t
		(t-s)^{-\frac{n}2(\frac{n+q}{2nq}-\frac1q)-\frac12}
		\|
			(
				v(s)-\mathrm{e}^{c_{\ast}s\Delta}v_0
			)
			\nabla K*v(s)
		\|_{\frac{2nq}{n+q}}
	\, \mathrm{d}s \\
	\le \, &
	C
	\int_0^t
		(t-s)^{-\frac34+\frac{n}{4q}}
		\|
			v(s)-\mathrm{e}^{c_{\ast}s\Delta}v_0
		\|_q
		\|
			\nabla K
			*
			v(s)
		\|_{\frac{2nq}{q-n}}
	\, \mathrm{d}s \\
	\le \, & 
	C
	\int_0^t
		(t-s)^{-\frac34+\frac{n}{4q}}
		\|
			v(s)-\mathrm{e}^{c_{\ast}s\Delta}v_0
		\|_q
		\|
			\nabla K
		\|_1
		\|
			v(s)
		\|_{\frac{2nq}{q-n}}
	\, \mathrm{d}s \\
	\le \, & 
	C
	F_q(t)
	\|
		v_0
	\|_n
	\int_0^t
		(t-s)^{-\frac34+\frac{n}{4q}}
		s^{-\frac34+\frac{n}{4q}}
		(1+s)^{-\sigma}
	\, \mathrm{d}s \\
	\le \, &
	C
	t^{-\frac{n}2(\frac1n-\frac1q)}
	(1+t)^{-\sigma}
	F_q(t)
	\|
		v_0
	\|_n
\end{align*}
for $t>0$. 
Thus this yields \eqref{Decay;v3v4}. 
\end{proof}

\begin{proof}[Proof of Proposition \ref{prop;t31}]
Let $n \le q \le \infty$ and $\sigma = \sigma_0$ for any fixed $0 < \sigma_0 < 1/4$. 
It follows from \eqref{asy;la} that  
the function $V_1(t)$ given by \eqref{eq;v1} is estimated as 
\[
	\|
		V_1(t)
	\|_q
	=
	\|\mathrm{e}^{t\mathcal{L}_A}v_0-\mathrm{e}^{c_{\ast} t\Delta}v_0\|_q\\
	\le 
	C
	t^{-\frac{n}2(\frac1n-\frac1q)}
	(1+t)^{-\sigma}
	\|
		v_0
	\|_n
\]
for $t>0$. 
Therefore, by \eqref{r;v}, \eqref{Decay;v2} and \eqref{Decay;v3v4}, we have
\begin{equation}\label{est;vh0}
\begin{split}
	t^{\frac{n}2(\frac1n-\frac1q)}
	(1+t)^{\sigma}
	\|
		v(t)
		-
		\mathrm{e}^{c_{\ast} t \Delta}
		v_0
	\|_q 
	\le \, &
	C_1
	(
	\|
		v_0
	\|_n
	+
	\|
		v_0
	\|_n^2
	)
	+
	C_2
	F_q(t)
	\|
		v_0
	\|_n
	\\
	\le \, &
	2
	C_1
	\|
		v_0
	\|_n
	+
	C_2
	F_q(t)
	\|
		v_0
	\|_n
\end{split}
\end{equation}
for $t > 0$ and $\| v_0 \|_n < \varepsilon_*$, 
where the constant $\varepsilon_\ast$ is the one given in Proposition \ref{prop;qgs}.
Taking the supremum over $0 < t \le \tau$ in \eqref{est;vh0} and using the fact that $F_q$ is nondecreasing in $t$ by definition,
we obtain
\begin{equation}\label{est;Fq}
	F_q(\tau) 
	\le 
	2
	C_1
	\|
		v_0
	\|_n
	+
	C_2
	F_q(\tau)
	\|
		v_0
	\|_n
	, \qquad \tau>0,
\end{equation}
where $C_1, C_2 > 0$ are independent of $v_0$ and $\tau$.

Choose 
$
	\varepsilon_0 
	=
	\varepsilon_0(q)
	\in
	(
		0,
		\min
		\{ 
			1, \varepsilon_*
		 \}
	]
$ 
so that 
$
	C_2 \varepsilon_0 \le 1/2
$.
Then, 
\eqref{est;Fq} gives
\[
	F_q(\tau)
	\le 
	2
	C_1
	\|
		v_0
	\|_n
	+
	\frac{1}{2}
	F_q(\tau), \quad \tau>0, 
\]
and therefore
$
	F_q(\tau) 
	\le 
	4
	C_1
	\|
		v_0
	\|_n
$ 
for $\tau > 0$. 
Consequently, 
\[
	\tau^{\frac{n}2(\frac1n-\frac1q)}
	\|
		v(\tau)
		-
		\mathrm{e}^{c_{\ast} \tau \Delta}
		v_0
	\|_q 
	\le 
	F_q(\tau)
	(1+\tau)^{-\sigma}
	\le 
	4
	(1+\tau)^{-\sigma}
	C_1
	\|
		v_0
	\|_n
	, \qquad \tau>0,
\]
and \eqref{Asy;mt} holds. 
\end{proof}

\subsection{Proof of Theorem \ref{t3}}

\begin{proof}[Proof of Theorem \ref{t3}]
Combining Propositions \ref{prop;t32} and \ref{prop;t31} proves Theorem \ref{t3}.
\end{proof}

\section*{Acknowledgments}
H. Wakui was supported by JSPS KAKENHI Grant Numbers JP23K19005 and JP25K07080.
T. Yamada was supported by JSPS KAKENHI Grant Number JP24K06806.

\section*{AI Use Disclosure}
During the preparation of this manuscript, the authors used
ChatGPT (OpenAI) to assist with language editing and rephrasing,
the identification of possible typographical and consistency
issues, and the formatting of bibliographic information.
The mathematical results and proofs were developed independently
by the authors before this editorial assistance.
Changes adopted in response to ChatGPT's suggestions were limited
to improving the presentation of this independently developed work.
The authors reviewed and verified all adopted changes and take
full responsibility for the accuracy and integrity of the manuscript.

\section*{Conflict of Interest Statement}
The authors declare no conflicts of interest.

\section*{Data Availability Statement}
Data sharing is not applicable---the paper describes entirely
theoretical research.

\bibliographystyle{abbrv}
\bibliography{ARKS_asymptotic_references}

\end{document}